\documentclass[reqno]{amsart}
\usepackage[T1]{fontenc}
\usepackage[latin1]{inputenc}
\usepackage[english]{babel}
\usepackage[babel]{csquotes}

\usepackage{cite}
\usepackage{amssymb}
\usepackage{amsmath}
\usepackage{amsthm}
\usepackage{latexsym}
\usepackage{graphicx}
\usepackage{mathrsfs}
\usepackage{mathrsfs}
\usepackage{bbm}
\usepackage{verbatim}

\usepackage[usenames,dvipsnames]{color} 
\newcommand{\UUU}{\color{black}}
\newcommand{\RRR}{\color{black}}
\newcommand{\EEE}{\color{black}}

\newtheorem{thm}{Theorem}[section]

\newtheorem{lem}[thm]{Lemma}
\newtheorem{prop}[thm]{Proposition}
\newtheorem{defin}[thm]{Definition}
\newtheorem{remark}[thm]{Remark}

\def\enne{\mathbb{N}}
\def\zeta{\mathbb{Z}}

\def\erre{\mathbb{R}}
\def\Rz{\mathbb{R}}
\def\Nz{\mathbb{N}}

\def\P{\mathbb{P}}

\def\E{\mathop{{}\mathbb{E}}}

\def\cL{\mathscr{L}}
\def\cF{\mathscr{F}}

\def\eps{\varepsilon}

\def\OO{\mathcal{O}}

\def\beq{\begin{equation}}
\def\eeq{\end{equation}}

\def\to{\rightarrow}
\def\wto{\rightharpoonup}
\def\wstarto{\stackrel{*}{\rightharpoonup}}
\def\embed{\hookrightarrow}
\def\cembed{\stackrel{c}{\hookrightarrow}}

\def\norm #1{\left\|#1\right\|}

\def\sp #1#2{\left<#1,#2\right>}
\newcommand\ip\sp
\def\qv #1{\left<#1\right>}

\begin{document}
\title[Doubly nonlinear stochastic evolution equations]{Doubly nonlinear stochastic evolution equations}

\author{Luca Scarpa}
\address[Luca Scarpa]{Faculty of Mathematics, University of Vienna, 
Oskar-Morgenstern-Platz 1, 1090 Wien, Austria.}
\email{luca.scarpa@univie.ac.at}
\urladdr{http://www.mat.univie.ac.at/$\sim$scarpa}

\author{Ulisse Stefanelli}
\address[Ulisse Stefanelli]{Faculty of Mathematics, University of Vienna, 
Oskar-Morgenstern-Platz 1, 1090 Wien, Austria  and  Istituto di Matematica
Applicata e Tecnologie Informatiche \textit{{E. Magenes}}, v. Ferrata 1, 27100
Pavia, Italy.}
\email{ulisse.stefanelli@univie.ac.at}
\urladdr{http://www.mat.univie.ac.at/$\sim$stefanelli}



\begin{abstract}
Nonlinear diffusion problems featuring stochastic effects may be
described by stochastic partial differential equations of the form 
$$d\alpha(u) - {\rm div}\,(\beta_1(\nabla u)) \, dt + \beta_0(u)\, dt\ni f(u) \, dt +
  G(u) \, dW.$$
We present an existence theory for such equations under general monotonicity assumptions on the nonlinearities.
In particular, $\alpha$, $\beta_0$, and $\beta_1$ are allowed to be
multivalued, as required by the modelization of second-order solid-liquid
transitions. In this regard, the equation corresponds to a
nonlinear-diffusion version of the classical two-phase Stefan
problem with stochastic perturbation.

The existence of martingale solutions is proved via
regularization and passage-to-the-limit.
The identification of the limit is obtained by a
lower-semicontinuity argument based on a suitably generalized It\^o's
formula.  Under some more restrictive assumptions on the
nonlinearities,  existence and
uniqueness of strong solutions follows.  Besides the relation
above, the theory covers equations with nonlocal
terms as well as systems. 
\end{abstract}

\maketitle


\section{Introduction}
\setcounter{equation}{0}
\label{sec:intro}

Nonlinear diffusion arises ubiquitously in
applications. Porous-media dynamics, the
Hele-Shaw cell, non-Newtonian fluids, and the Mean-Curvature flow are examples of relevant
applied settings falling into this class. The corresponding partial
differential models generally feature strong
nonlinearities, also of degenerate or singular type, which in turn
pose severe analytical challenges.

In the paper we focus on a class of stochastic partial differential
equations (SPDEs) arising in nonequilibrium
thermodynamics and describing solid-liquid phase
transitions. In particular, we are concerned with SPDEs of the form 
\begin{equation}
  \label{eq:00}
  d\alpha(u) - {\rm div}\,(\beta_1(\nabla u)) \, dt + \beta_0(u)\, dt\ni f(u) \, dt +
  G(u) \, dW . 
\end{equation}
The scalar-valued $u$ models the {\it temperature} of a specimen
and $\alpha(u)$ with $\alpha : \Rz\to 2^\Rz$ maximal monotone represents the
{\it internal energy} of the system. Note that the latter is possibly
multivalued, modeling the case of so-called {\it second-order} phase
change, namely the occurrence of a latent heat of transformation. The field $\beta_1(\nabla u)$ with $\beta_1
: \Rz^d \to 2^{\Rz^d}$ maximal monotone is a generalized nonlinear
{\it heat flux},  $-\beta_0(u)+f(u)$ with $ \beta_0 : \Rz\to 2^\Rz$ maximal
monotone and $f$ Lipschitz-continuous is a {\it deterministic} temperature-dependent
forcing, while the term $G(u)\, dW$ models a {\it stochastic} temperature-dependent
forcing instead. Here, $W$ is a 
suitably defined Wiener process, see below.


Under different choices for the nonlinearities, relation
~\eqref{eq:00} arises in relation with different models for \EEE nonequilibrium
thermodynamics. In particular, by letting the graph $\alpha$ include a vertical
segment through the origin, equation \eqref{eq:00} corresponds to a
nonlinear-diffusion version of the two-phase Stefan
problem with stochastic perturbation. 
More specifically, one can choose $\alpha=\operatorname{sign}+\tilde\alpha$, 
where $\tilde\alpha:\erre\to\erre$ is a nondecreasing continuous function and
$\operatorname{sign}:\erre\to2^\erre$ is the sign-graph, defined as
$\operatorname{sign}(r) = r/|r|$ for $r\not = 0$ and
$\operatorname{sign}(0) = [-1,1]$. By allowing noncoercive graphs
$\alpha$, an option which is
however not covered by our analysis, relation \eqref{eq:00} arises in
connection with the Hele-Shaw cell and filtration through porous
media as well.
The reader is referred to {\sc
Visintin} \cite{Visintin96} for a discussion on the relevance of
relation \eqref{eq:00} in the frame of phase-transition modeling.

The focus of the paper is on the existence of solutions for
SPDEs of the class of relation \eqref{eq:00}.
In particular, we focus on a variational reformulation of the
initial-boundary value problem for 
\eqref{eq:00} in terms of the abstract stochastic doubly nonlinear
equation
\begin{equation}\label{eq:0}
  d(Au) + Bu\,dt \ni F(u)\,dt + G(u)\,dW\,, \qquad (Au)(0)\ni v_0\,.
\end{equation}
Here, the process $u$ takes values in the real,
separable Hilbert space $V$, which embeds compactly and densely into
a second Hilbert space $H$. 
The operators $A=\partial \varphi:H\to 2^H$ and $B:V \to 2^{V^*}$ 
are maximal monotone, possibly multivalued, coercive and linearly
bounded. The operator $A$ is assumed to be cyclic monotone, with
G\^ateaux-differentiable inverse $A^{-1}$ (see below), and the map $F:H \to H$ is
Lipschitz-continuous. Eventually, $W$ is a cylindrical Wiener process
on a third, separable Hilbert space $U$ and the time-dependent
operator $G$ takes values in the space of Hilbert-Schmidt
operators from $U$ to $H$. The term {\it doubly nonlinear} refers here
to the fact that both operators $A$ and $B$ are nonlinear.

In the deterministic case $G=0$, problem \eqref{eq:0}
is classical. Its analysis can be traced back at least to {\sc Bardos \&
Brezis} \cite{BB} and  {\sc
Raviart} \cite{Raviart}. Other early contributions are from {\sc
Grange \& Mignot} \cite{Grange-Mignot72}, {\sc Barbu} \cite{Barbu75},
{\sc DiBenedetto \& Showalter} \cite{diben-show}, {\sc Alt \&
Luckhaus} \cite{Alt-Luckhaus}, and {\sc Bernis} \cite{Bernis88}. For
a collection of further developments, the reader can check 
\cite{A06,AH04,AH04b,akst4,GS04,Gilardi08,Hokkanen91,Hokkanen92b,Hokkanen92,MW-dn,vol,fico,ganzo},
among many others. 

In the stochastic case $G\not =0$, problem \eqref{eq:0} is well-studied
for $A$ linear and not degenerate. The reader is referred to
the seminal contributions by {\sc Pardoux} \cite{Pard0,Pard} and {\sc Krylov
\&  Rozovski\u \i} \cite{kr}, and to the monographs \cite{dapratozab,pr},
for a general overview. 
In the context of variational approach, problem \eqref{eq:0} with $A$
linear has been proved by {\sc Gess} \cite{Gess} to admit 
strong solutions for $B$ cyclic monotone and subhomogeneous.
Well-posedness from
a variational approach have been obtained also
under no growth condition on the drift
in \cite{mar-scar-diss, mar-scar-ref, mar-scar-erg} for semilinear equations,
in \cite{mar-scar-note, mar-scar-div, scar-div} for equations in divergence form,
and in \cite{barbu-daprato-rock,orr-scar,scar-SCH} for 
porous-media, Allen-Cahn, and Cahn-Hilliard equations.

Alternatively to the variational approach, the analysis of stochastic evolution equations 
in the form \eqref{eq:0} (in the case $A$ linear)
have been developed in several directions.
First of all, in the classical work
by {\sc Bensoussan \& Rascanu} \cite{BenRas},
existence of strong and martingale solutions   in terms
of stochastic variational inequalities 
is proved.
Stochastic variational inequalities have then been
used to formulate weaker concepts of solutions
also for other types of equations
as divergence-form equations   \cite{GessRoc}
or fast-diffusion equations \cite{GessRoc-fast}.
More recently, an operator approach to
monotone equations with maximal monotone drift
and linear multiplicative noise has been given 
by {\sc Barbu \& R\"ockner} in \cite{BarRoc-op}:
here solutions are defined using a suitable 
rescaling argument and monotonicity 
techniques in spaces of stochastic processes.

The case of $A$ nonlinear and $B$ linear has been originally treated by {\sc Barbu \& Da~Prato}
\cite{Barbu-DaPrato} (see also \cite{BBT14,KM16}) in
the framework of the two-phase stochastic Stefan problem. 
There, the authors study equation \eqref{eq:00}
with the choices $\beta_1=I$, $\beta_0=f=0$, and
$\alpha=\operatorname{sign}+I$, with $I$ being
the identity on $\Rz$. Existence and uniqueness of strong 
solutions is obtained through a suitable change of variable,
rewriting the equation in the dual space $H^{-1} $,
and using regularization and passage to the limit techniques. Such
reformulation hinges on the linearity of $\beta_1=I$.

The only contribution tackling the genuinely doubly
nonlinear case is {\sc Sapountzoglou,  Wittbold,
\&~Zimmermann} \cite{SWZ18}, where nonetheless $A$
is assumed to be Lipschitz-continuous, $B$ has the specific
divergence form of \eqref{eq:00}, and $F=0$. 
By contrast, these assumptions are dropped here. In
particular, we stress that
in our analysis
$A$ is allowed to be multivalued, as for
$\alpha=\operatorname{sign}+\tilde\alpha$ mentioned above.
As far as the operator $B$ is concerned, 
our techniques rely on the Hilbert stricture of the space $V$
and require a linear growth assumption on $B$,
while in \cite{SWZ18} also $p$-growth conditions 
for $B$ are included. In this direction, 
let us mention that in order to study equations in the form
\eqref{eq:0} with both $A$ multivalued and 
$B$ possibly of $p$-growth (as for example 
in the case of pseudo monotone operators)
one should set the problem in a Banach space framework for $V$.
This requires completely different techniques than the 
ones used in this paper, and it is currently studied 
in a work in progress. 

Our setting is exactly that of {\sc Di Benedetto \& Showalter}
\cite{diben-show}, whose findings we extend here to the stochastic case. 
In particular, note that doubly nonlinear equations in the form
\eqref{eq:0} cannot be treated using existing techniques. Indeed, 
the intuitive substitution $v=A(u)$ does not work, 
as the operator $B\circ A^{-1}$ is not well-defined on $V^*$
due to the lack of coercivity of $A$ on $V$. For this reason, 
equations in the form \eqref{eq:0} are usually referred to as
{\it implicit}, in the sense that they {\em cannot} be rewritten in the form
$dv + \tilde B v\,dt = G\,dW,$
for any suitable choice of the operator $\tilde B$.

Our main result is the existence of
martingale solutions to problem \eqref{eq:0} (Theorem \ref{th:1}). These are
obtained via a regularization and passage-to-the-limit procedure. The
identification of the nonlinear term $Au$ in the limit directly
follows from a compactness argument, based on the linear boundedness of $A$ in
the intermediate space $H$. On the other hand, the limiting
$Bu$ is identified by a semicontinuity
argument, which in turn hinges on the availability of an It\^o
formula for $\varphi^*$  (Proposition
\ref{prop:ito}), where $\partial \varphi^*=
A^{-1}$.
This however does not
fall within the framework of {\sc Pardoux} \cite{Pard}
or {\sc Da Prato \& Zabczyk} \cite{dapratozab} due to the nonlinearity 
of $A$. Note that 
well-known approximation techniques
based on regularization through linear smoothing operators
are ineffective in our framework, since they are in general
not compatible with the nonlinearity $A$.
Such difficulties are overcome by an {\em ad-hoc} 
regularization based on smoothing nonlinear elliptic operators:
this procedure turns out to be effective for our purpose, 
and requires a specific and detailed asymptotic analysis.
For the validity of It\^o's formula, whose proof represents the
technical core of the paper, one has to ask $A^{-1}$ to be
well-behaved. In particular, $A^{-1}$ is here assumed to be 
G\^ateaux differentiable and its differential $D(A^{-1})$ to be smooth
enough. These smoothness assumptions are nonetheless fulfilled in the
case of \eqref{eq:00} whenever  
$\alpha^{-1}$ is smooth enough
(see Section \ref{sec:appl}). 
Let us 
point out that such condition is satisfied 
also when $\alpha$ has
the form $\alpha=\operatorname{sign}+\tilde\alpha$, as in the doubly
nonlinear stochastic Stefan problem. In particular, $\alpha$ needs not
be Lipschitz-continuous nor single-valued.

Exactly as in the deterministic situation
\cite{diben-show},
in case $A$ or $B$ is linear, continuous, and symmetric, one can prove that martingale solutions are
unique. It hence turns out that they are 
also strong in probability (Theorem \ref{th:2}).
The conditions ensuring uniqueness are essentially sharp,
in the sense that if any of them is not satisfied then nonuniqueness
of solutions to problem
\eqref{eq:0} may occur, 
even if $F=G=0$ and $V=H=\erre$.

Let us now briefly summarize how the paper is structured.
We fix the setting and state our main results in Section~\ref{sec:main}. 
A collection of preliminary observations on the
approximation of the nonlinear operators is recorded
in Section~\ref{sec:prelim}. Then, Section~\ref{sec:ito} is devoted to
the proof of the above-mentioned It\^o formula. This is used in
Section~\ref{sec:exist} in order to prove the existence of martingale
solutions, namely Theorem~\ref{th:1}. Theorem~\ref{th:2} on existence
and uniqueness of strong solutions in probability is then proved in
Section~\ref{sec:uniq_strong}. Moving from the abstract theory,
in Section~\ref{sec:appl} we discuss the existence of solutions to
SPDEs of the
form \eqref{eq:00}. In addition, classes of SPDEs with nonlocal terms
and of SPDE systems are also proved to be solvable.
 

\section{Setting and statement of the main results}
\label{sec:main}

Let $(\Omega,\cF, \P)$ be a probability space endowed with 
a filtration $(\cF_t)_{t\in[0,T]}$ which is complete and right-continuous, 
where $T>0$ is a fixed final time. 

For any Banach space $E$ we shall use the usual symbols 
$L^p(\Omega; E)$, $L^p(0,T; E)$ and $C^0([0,T]; E)$
for the spaces of $p$-Bochner integrable $E$-valued functions on $\Omega$
and $(0,T)$, and for the space of continuous functions $[0,T]\to E$, respectively.
We will also use the symbol $L^0(\Omega; E)$ for the space of 
measurable functions from $(\Omega,\cF)$ to $E$.
Furthermore, if $E_1$ and $E_2$ are Banach spaces, the symbols $\cL(E_1, E_2)$,
$\cL_s(E_1,E_2)$ and $\cL_w(E_1,E_2)$
denote the space of linear continuous operators from $E_1$ to $E_2$ endowed with 
the norm topology, strong operator topology, or weak operator topology, respectively.
If $E_1$ and $E_2$ are Hilbert spaces, we shall also use $\cL^1(E_1,E_2)$
and $\cL^2(E_1,E_2)$ to indicate the spaces of trace-class and Hilbert-Schmidt
operators from $E_1$ to $E_2$, respectively.

Let $W$ be a cylindrical Wiener process
on a separable Hilbert space $U$. This amounts to saying that
$W$ is {\em formally} defined as the infinite sum
\[
  W(t)=\sum_{k=0}^{\infty}\beta_k(t) e_k\,, \qquad t\in[0,T]\,,
\]
where $(e_k)_k$ is a complete orthonormal system in $U$ and $(\beta_k)_k$
are real-valued independent Brownian motions.

Let $V$ and $H$ be separable Hilbert spaces such that $V\embed H$
densely, continuously and compactly. By identifying $H$ with its dual $H^*$,
$(V,H,V^*)$ turns out to be a classical Hilbert
  triplet. In particular,
\[
  V  \cembed H \cembed V^*\,,
\]
where all inclusions are dense, continuous, and compact. 
Norms, scalar products, and dualities will be denoted by the symbols
$\norm{\cdot}$, $(\cdot,\cdot)$, and $\ip{\cdot}{\cdot}$, respectively, with a sub-script
specifying the spaces in consideration. We shall denote by
$R:V\to V^*$ the Riesz isomorphism of $V$ and define the Hilbert space 
\[
  V_0:=\left\{x\in V: \; Rx\in H\right\}\,, \qquad
  \norm{x}_{V_0}^2:=\norm{x}_V^2 + \norm{Rx}_H^2\,, \quad x\in V_0\,.
\]
Note that the inclusion $V_0\embed V$ is compact and dense. Indeed,
if $(x_n)_n\subset V_0$, $x\in V_0$ and $x_n\wto x$ in $V_0$, then
by compactness of $V$ in $H$ we have $x_n\to x$ strongly in $H$; since also
$Rx_n\wto Rx$ in $H$, we infer that 
\[
  \norm{x_n}_V^2=\ip{Rx_n}{x_n}_V=(Rx_n,x_n)_H \to (Rx, x)_H=\norm{x}_V^2\,,
\]
so that $x_n\to x$ in $V$ strongly, and $V_0\cembed V$ compactly. Moreover, 
it is also not difficult to check that
$V_0$ is dense in $V$.

The following assumptions will be in order throughout the work:
\begin{description}
  \item[(H1)] $A=\partial \varphi$, where $\varphi:H\to[0,+\infty)$
    is proper, convex, and lower-semicontinuous, $A(0)\ni0$, 
	and there exists $C_A>0$ such that 
	\[
	\norm{y}_H\leq C_A(1+\norm{x}_H) \qquad\forall\,x\in H\,,\quad\forall\,y\in A(x)\,.
	\]
	For any $\eps>0$, the $\eps$-Yosida approximation of $A$ will
        be denoted by $A^\eps$, namely $A^\eps:= (I-(I +\eps
        A)^{-1})/\eps$ where $I$ is the identity on $H$ (see \cite{brezis}), and 
	we assume that 
	\[
	(A^\eps(x),Rx)_H\geq 0 \qquad\forall\,x\in V_0\,,\quad\forall\,\eps>0\,.
        \] 
  \item[(H2)] $A$ is strongly monotone on $H$, i.e.~there exists $c_A>0$ such that 
  	$$
	(y_1-y_2, x_1-x_2)_H\geq c_A\norm{x_1-x_2}_H^2 \qquad\forall\,x_i\in D(A)\,, \quad
	\forall\,y_i\in A(x_i)\,, \quad i=1,2\,.
	$$
	Since $A(0)\ni0$, this implies in particular that $A$ is coercive on $H$, hence also surjective
	by maximal monotonicity, and that the inverse
  	operator $A^{-1}:H\to H$ is well-defined and Lipschitz-continuous.
  \item[(H3)] $A^{-1}:H\to H$ is G\^ateaux-differentiable and there exists
        a Banach space $Y\embed H$ continuously and densely such that
        \begin{align*}
        &D(A^{-1})\in C^0(H; \cL_w(H,H)\cap \cL(Y,H))\,,\\
        &\left[I + D(A^{-1})((I+A^{-1})^{-1}x)\right]^{-1}\in \cL(V,Y) \quad\forall\, x\in V\,.
        \end{align*}
        Note that such assumption implies that $(\varphi^*)_{|V}\in C^2(V)$, so that 
        in particular $D(A^{-1})(y)$ is symmetric for all $y\in V$.
  \item[(H4)] for any family $(x_\eps)_{\eps>0}\subset V$, $x\in V$, and $y\in A(x)$ 
  such that $x_\eps\wto x$ in $V$ and $A^\eps(x_\eps)\wto y$ in $H$ as $\eps\searrow0$, 
  it holds
  \[
  D(A^{-1})(A^\eps(x_\eps))\to D(A^{-1})(y) \qquad\text{in } \cL_s(H,H)\,.
  \]
  This assumption is of a technical nature and has to be checked 
  in each specific problem. Note nonetheless that it is satisfied in
  several relevant situations (see Section~\ref{sec:appl}
  for some concrete examples).
  \UUU
  \item[\RRR(H5)\EEE] there exists a separable Hilbert space $Z\subset V$, densely embedded in $H$,
a constant $\eta\in (1/3,1/2)$, and an increasing function $f:[0,+\infty)\to[0,+\infty)$
such that, for every $x\in V$ it holds that 
\[
D(A^{-1})(A^{\eps}(x)) \in \mathscr L(Z, D(R^\eta)) \qquad\forall\,\eps>0
\]
and
\[
  \norm{R^\eta D(A^{-1})(A^{\eps}(x))h}_H \leq f(\norm{x}_V)\norm{h}_Z
  \qquad\forall\,h\in Z\,,\quad\forall\,\eps>0\,.
\]
\EEE
  \item[\RRR(H6)\EEE] $B:V\to 2^{V^*}$ is maximal monotone and there exists $C_B, c_B>0$ such that 
  	\begin{align*}
&	\norm{y}_{V^*}\leq C_B\left(1+\norm{x}_V\right)  \ \
  \text{and} \ \ 	\ip{y}{x}_V\geq c_B\norm{x}_V^2  \qquad\forall\, x\in V\,, \quad\forall\,y\in B(x)\,.
	\end{align*}
  \item[\RRR(H7)\EEE] $F:[0,T]\times H\to H$ is measurable with
  $F(\cdot,0)\in L^2(0,T; H)$ and
  there exists $L_F>0$ such that,
  for almost every $t\in (0,T)$,
  \[
    \norm{F(t,x_1)-F(t,x_2)}_{H}\leq L_F\norm{x_1-x_2}_H \qquad
    \forall\,x_1,x_2\in H\,.
  \]
  \item[\RRR(H8)\EEE] $G:[0,T]\times H\to\cL^2(U,H)$ is measurable and there exists $L_G>0$
  such that, for every $t\in[0,T]$ and $x_1,x_2,x\in H$,
  \begin{align*}
  \norm{G(t,x_1)-G(t,x_2)}_{\cL^2(U,H)}&\leq L_G\norm{x_1-x_2}_H\,,\\
  \norm{G(t,x)}_{\cL^2(U,H)}&\leq L_G(1+\norm{x}_H)\,.
  \end{align*}
  \item[\RRR(H9)\EEE] $v_0\in L^q(\Omega,\cF_0; H)$, $\varphi^*(v_0)\in L^{q/2}(\Omega,\cF_0)$,
  $u_0:=A^{-1}(v_0)\in L^q(\Omega,\cF_0; V)$, for $q>2$.
\end{description}

Let us comment now on assumptions {(H1)--\RRR(H6)\EEE}, pointing out their major consequences
and giving some sufficient conditions for these to hold. For explicit examples
of operators $A$ and $B$ we refer to Section~\ref{sec:appl}.

\begin{remark}[Hypothesis {(H1)}]\rm
	Assumption {(H1)} is very common in the context of doubly nonlinear evolution equations:
	see for example \cite{diben-show}. In particular, it
        entails that
	$A_{|V}:V\to 2^{V^*}$ is maximal monotone.  
	Indeed, the monotonicity is trivial. As for the maximality, note that for any $\eps>0$
	the Yosida approximation $A^\eps:H\to H$ is Lipschitz-continuous, so that 
	for every $y\in V^*$ there is a unique $x_\eps\in V$ such that 
	\[
 	R x_\eps + A^\eps x_\eps = y\,.
	\]
	Testing by $x_\eps\in V$, it follows that
	\[
	  \norm{x_\eps}_V^2 + (A^\eps x_\eps, x_\eps)_H = 
	  \ip{y}{x_\eps}_V\leq \frac12\norm{x_\eps}_V^2
	  +\frac1{2}\norm{y}_{V^*}^2\,,
	\]
	so that $(x_\eps)_\eps$ is bounded in $V$. Recalling {(H1)} 
	we deduce that $(A_\eps x_\eps)_\eps$
	is bounded in $H$. Hence, there are $x\in V$ and $z\in H$ such that, as $\eps\searrow0$,
	$x_\eps\wto x$ in $V$, $x_\eps\to x$ in $H$, and $A^\eps x_\eps\wto z$ in $H$, 
	which yield $z\in Ax$ by strong-weak closure of $A$ in $H\times H$. Letting $\eps\searrow0$
	we also deduce that $Rx + z = y$, from which we conclude
        that $A_{|V}:V\to 2^{V^*}$ is maximal.
        Since one readily has that
$A_{|V}\subseteq\partial\varphi_{|V}: V \to 2^{V^*}$,
	by maximality of $A_{|V}$ it holds $A_{|V}=\partial\varphi_{|V}$. 
\end{remark}

\begin{remark}[Hypothesis {(H2)}]\rm
	Note that, although $A^{-1}:H\to H$ is
        Lipschitz-continuous by {(H2)}, 
	 $A$ can still be multivalued.
	A relevant class of strongly monotone operators $A$ is 
	given by those of the form $A=cI + \tilde A$, where $\tilde A$ is maximal monotone
	on $H$ and $c>0$: see Section~\ref{sec:appl}.
\end{remark}

\begin{remark}[Hypothesis {(H3)}]\rm
	We are requiring that $A^{-1}$ is G\^ateaux-differentiable on the whole space $H$,
	with G\^ateaux derivative $D(A^{-1})$ continuous from $H$ to $\cL_w(H,H)$,
	and Fr\'echet-differentiable with continuous derivative in $Y$. Of course, if $A^{-1}\in C^1(H,H)$
	these conditions are easily satisfied. However, in many applications the operator $A^{-1}$
	{\em is not} Fr\'echet-differentiable in $H$, so it is important to require less stringent assumptions
	as in {(H3)}. For example, if $A$ is the Nemitzsky operator associated to a maximal monotone graph 
	$\alpha$ on $\erre$, 
	it is well-known that $A^{-1}$ is always G\^ateaux differentiable as soon as $\alpha^{-1}$ is 
	$C^{1,1}$, and is Fr\'echet-differentiable if and only 
	if the $\alpha^{-1}$ is an affine function, which is clearly a too
	restrictive condition. Further details are given in
	Section~\ref{sec:appl}.
\end{remark}

\begin{remark}[Hypothesis {(H4)}]\rm
	If $A$ is the Nemitzsky operator associated to a maximal monotone graph 
	$\alpha$ on $\erre$, as in the framework of problem \eqref{eq:00},
	it can be seen that (H4) is always satisfied when $\alpha$ is a continuous
	function (single-valued), not necessarily
        Lipschitz-continuous. Additionally, one could consider {\em
          multivalued} graphs $\alpha$ as well, see
Section~\ref{sec:appl} for details.
\end{remark}

\begin{remark}[Hypothesis {\RRR(H6)\EEE}]\rm
In  \cite{SWZ18} the operator $B$ is just required to be polynomially
bounded, which calls for framing the problem in a Banach-space
setting. We impose a linear bound on $B$ instead, see \RRR(H6)\EEE, which
allows a formulation in Hilbert spaces.
\end{remark}

Let us now state the concepts of strong and martingale solution for
problem \eqref{eq:0}.

\begin{defin}[Strong solution]
  A \emph{strong solution} to \eqref{eq:0} is a triple $(u,v,w)$ of
  progressively measurable processes with values in $V$, $H$, and $V^*$, respectively, 
  such that 
  \begin{align*}
  &u \in L^0(\Omega; L^2(0,T; V))\,, \\ 
  &v \in L^0(\Omega; L^2(0,T; H)\cap C^0([0,T]; V^*))\,, \\ 
  &w\in L^0(\Omega; L^2(0,T; V^*))\,,\\
  &v\in Au\,, \quad w\in Bu \qquad\text{a.e.~in } \Omega\times(0,T)\,,\\
  &v(t)+\int_0^tw(s)\,ds = v_0 +
  \int_0^t F(s,u(s))\,ds
  + \int_0^tG(s,u(s))\,dW(s) \quad\text{in }V^* \quad\forall\,t\in[0,T]\quad\P\text{-a.s.}
  \end{align*}
\end{defin}

\begin{defin}[Martingale solution] A \emph{martingale solution} to \eqref{eq:0} is a quintuplet 
  $$((\hat\Omega, \hat\cF, (\hat\cF_t)_{t\in[0,T]}, \hat\P),\hat W, \hat u, \hat v, \hat w),$$
  where $(\hat\Omega, \hat\cF, \hat\P)$ is a probability space endowed with a filtration
  $(\hat\cF_t)_{t\in[0,T]}$ which is saturated and right-continuous, $\hat W$ is
  a $(\hat\cF_t)_t$-cylindrical Wiener process on $U$, and $\hat u$, $\hat v$, and $\hat w$
  are progressively measurable processes with values in $V$, $H$, and $V^*$, respectively, such that 
  \begin{align*}
 & \hat u \in L^0(\hat\Omega; L^2(0,T; V))\,,\\
 & \hat v \in L^0(\hat\Omega; L^2(0,T; H)\cap C^0([0,T]; V^*))\,, \\ 
 & \hat w\in L^0(\hat\Omega; L^2(0,T; V^*))\,,\\
 & \hat v\in A\hat u\,, \quad \hat w\in B\hat u \qquad\text{a.e.~in } \hat\Omega\times(0,T)\,,\\
 & \hat v(t)+\int_0^t\hat w(s)\,ds = 
  \hat v(0) + 
  \int_0^t F(s,\hat u(s))\,ds+
  \int_0^tG(s,\hat u(s))\,d\hat W(s) \quad\text{in }V^* \quad\forall\,t\in[0,T]\quad\hat\P\text{-a.s.}
  \end{align*}
  and $\hat v(0)$ has the same law of $v_0$ on $V^*$.
\end{defin}

The main results of this work are the following.
\begin{thm}[Existence of martingale solutions]
  \label{th:1}
  Assume {\em{(H1)--\RRR(H9)\EEE}}. Then problem \eqref{eq:0} admits a
  martingale solution which additionally satisfies
  \begin{align*}
 & \hat u \in L^q(\hat\Omega; L^\infty(0,T; H)\cap L^2(0,T; V))\,,\\
 & \hat v \in L^q(\hat\Omega; L^\infty(0,T; H)\cap C^0([0,T]; V^*))\,,\\
 & \hat w\in L^q(\hat\Omega; L^2(0,T; V^*))\,.
  \end{align*}
\end{thm}

\begin{thm}[Existence and uniqueness of strong solutions]
  \label{th:2}
  Assume {\em{(H1)--\RRR (H9)\EEE}} and that the initial datum
  $v_0\in H$ is nonrandom. 
  If either $A$ or $B$ is
  linear, continuous, and symmetric,
  then problem \eqref{eq:0} admits a unique strong solution $(u,v,w)$ with 
  \begin{align*}
   &u \in L^q(\Omega; L^\infty(0,T; H)\cap L^2(0,T; V))\,,\\
  & v \in L^q(\Omega; L^\infty(0,T; H)\cap C^0([0,T]; V^*))\,,\\
  & w\in L^q(\Omega; L^2(0,T; V^*))\,.
  \end{align*}
\end{thm}

\begin{remark}\rm
  Note that the conditions on $A$ and $B$ ensuring uniqueness are sharp,
  in the sense that they cannot be weakened, even for $H=\erre$:
  see \cite{diben-show} for details.
\end{remark}


\section{Preliminary results}
\label{sec:prelim}

We collect in this section some auxiliary results that will be used 
throughout the work.

Recall that, for all $\eps\in(0,1)$, $A^\eps$ is the
$\eps$-Yosida approximation of $A$, namely, $A^\eps:= (I-(I +\eps
A)^{-1})/\eps$ where $I$ is the identity on $H$. We can
equivalently rewrite $A^\eps$ as 
\[
  A^\eps:=(\eps I + A^{-1})^{-1}:H\to H\,.
\]
Note that the operator $\eps I+A^{-1}:H\to H$ is maximal monotone, 
Lipschitz-continuous, and coercive on $H$, hence invertible. In order
to prove that the two definitions coincide, fix $x\in H$, and define
$y^\eps:=(\eps I + A^{-1})^{-1}(x)$ so that
$\eps y^\eps + A^{-1}(y^\eps)=x$. Setting also 
$x^\eps:=A^{-1}(y^\eps)$ we infer that 
$x^\eps + \eps A(x^\eps)\ni x$, so that $x^\eps=J_A^\eps(x)$,
where $J_A^\eps:=(I+\eps A)^{-1}:H\to H$ is the resolvent of~$A$.
Eventually, we deduce that $y^\eps=(x-x^\eps)/\eps$, as desired.

We collect some useful properties of $A^\eps$ in the following lemma.
\begin{lem}[Properties of $A^\eps$]
  \label{lem:A_eps}
  Let $\eps\in(0,c_A^{-1})$. Then the following properties hold:
  \begin{description}
  \item[(P1)] for every $x_1,x_2,x\in H$ it holds
  \begin{align*}
  &\left(A^\eps(x_1)-A^{\eps}(x_2), x_1-x_2\right)_H\geq \frac{c_A}2\norm{x_1-x_2}_H^2\,,\\
  &\norm{A^\eps(x_1)-A^\eps(x_2)}_H\leq\frac1\eps\norm{x_1-x_2}_H\,,\\
  &\norm{A^\eps(x)}_H\leq C_A\left(1+2\norm{x}_H\right)\,.
  \end{align*}
  \item[(P2)] there exists a convex function $\varphi^\eps:H\to[0,+\infty)$ with $\varphi^\eps\in C^1(H)$,
  $D\varphi^\eps=A^\eps$ and
  \[
  (\varphi^\eps)^*(y)=\frac\eps2\norm{y}_H^2 + \varphi^*(y) \qquad\forall\,y\in H\,.
  \]
  \item[(P3)] $A^\eps$ is G\^ateaux-differentiable and
  $D(A^\eps)\in C^0(H; \cL(V,H))$. 
  \end{description}
\end{lem}
\begin{proof}
  Ad {{(P1)}.}
  Since $A^{-1}:H\to H$ is monotone and Lipschitz continuous, the operator
  $\eps I+A^{-1}$ is monotone, coercive, and Lipschitz-continuous. Hence, $A^\eps:H\to H$
  is well-defined, monotone, and Lipschitz-continuous as well. Moreover, 
  for every $x_1,x_2\in H$, setting $y_1:=A^{\eps}(x_1)$ and $y_2:=A^\eps(x_2)$, 
  by strong monotonicity of $A$ we have
  \begin{align*}
  &\left(A^\eps(x_1)-A^{\eps}(x_2),x_1-x_2\right)_H\\
  &\quad=\left(y_1-y_2, \eps(y_1-y_2)+A^{-1}(y_1)-A^{-1}(y_2)\right)_H\\
  &\quad= \eps\norm{y_1-y_2}_H^2 + \left(y_1-y_2,A^{-1}(y_1)-A^{-1}(y_2)\right)_H\\
  &\quad\geq \eps\norm{y_1-y_2}_H^2 + c_A\norm{A^{-1}(y_1)-A^{-1}(y_2)}_H^2\\
  &\quad= \eps\norm{y_1-y_2}_H^2 + c_A\norm{x_1-x_2-\eps(y_1-y_2)}_H^2\,.
  \end{align*}
  Noting that $\eps c_A\leq 1$ by assumption, we also have
  \begin{align*}
  &\frac{c_A}2\norm{x_1-x_2}_H^2\leq c_A\norm{x_1-x_2-\eps(y_1-y_2)}_H^2+c_A\eps^2\norm{y_1-y_2}_H^2\\
  &\quad\leq c_A\norm{x_1-x_2-\eps(y_1-y_2)}_H^2+\eps\norm{y_1-y_2}_H^2\,,
  \end{align*}
  so that the strong monotonicity of $A^\eps$ follows. The
  Lipschitz-continuity of $A^\eps$ is well-known (see e.g.~\cite{brezis}).   
  Finally, if $x\in H$ and $y:=A^\eps(x)$, we have
  $\eps y+ A^{-1}(y)=x$, so that $y \in A(x-\eps y)$: hence, by (H1) 
  and the already proved Lipschitz-continuity we have
  $\norm{y}_H\leq\frac1\eps\norm{x}_H$, yielding
  \[
  \norm{y}_H\leq C_A\left(1+\norm{x-\eps y}_H\right)\leq C_A(1+\norm{x}_H+\eps\norm{y}_H)
  \leq C_A\left(1 + 2\norm{x}_H\right)\,,
  \]
  from which the linear growth condition of $A^\eps$ follows, uniformly in $\eps$.
  
  Ad {{(P2)}.} It is well-known that, by 
  defining $\varphi^\eps$
  as the Moreau-Yosida regularization of $\varphi$, namely 
 $$\varphi^\eps(x) := \inf_{v\in H} \left(\frac{\norm{x-v}_H^2}{2\eps} + \varphi(v)\right)\,,$$
 we have 
  $\partial \varphi^\eps=A^\eps$ and $\partial(\varphi^\eps)^*=(A^\eps)^{-1}$.
  Since $A^\eps$ is Lipschitz-continuous, we deduce that actually
  $\varphi^\eps\in C^1(H)$ and $D\varphi^\eps=A^\eps$, as required.

  Ad  {{(P3)}.} Let us show that $A^\eps$ is G\^ateaux differentiable.
  By {(H3)}, $(A^\eps)^{-1}=\eps I + A^{-1}$ is G\^ateaux differentiable
  and, for all $y\in H$, $D((A^{\eps})^{-1})(y)=\eps I +D(A^{-1})(y)$ is a linear isomorphism of $H$.
  Hence, this implies that $A^\eps$ is G\^ateaux-differentiable and its differential is given by
  \[
    D(A^\eps)(x)=\left(\eps I + D(A^{-1})(A^\eps(x))\right)^{-1}\in \cL(H,H)\,, \quad x\in H\,.
  \]
  Let us show that $D(A^\eps)\in C^0(H;\cL(V,H))$.
  Let $(x_n)_n\subset H$, $x\in H$ with 
  $x_n\to x$ in $H$, and $z\in V$ be arbitrary. Setting
  \[
  h_n:=D(A^{\eps})(x_n) z\,, \qquad h:=D(A^\eps)(x)z\,,
  \]
  we have that
  \[
  z=\eps h_n + D(A^{-1})(A^\eps(x_n))h_n\,, \qquad
  z=\eps h + D(A^{-1})(A^\eps(x))h\,,
  \]
  from which 
  \[
  \eps(h_n-h) + D(A^{-1})(A^\eps(x_n))(h_n-h) = \left(D(A^{-1})(A^\eps(x))-D(A^{-1})(A^\eps(x_n))\right)h\,.
  \]
  Hence, testing by $h_n-h$, using the monotonicity of $A^{-1}$ we have
  \begin{align*}
  &\eps\norm{h_n-h}_H^2\leq\left(\left(D(A^{-1})(A^\eps(x))-D(A^{-1})(A^\eps(x_n))\right)h,h_n-h\right)_H\\
  &\quad\leq\frac\eps2\norm{h_n-h}_H^2 
  + \frac1{2\eps}\norm{\left(D(A^{-1})(A^\eps(x))-D(A^{-1})(A^\eps(x_n))\right)h}_H^2\,,
  \end{align*}
  yielding
  \[
  \norm{h_n-h}_H\leq\frac1\eps\norm{\left(D(A^{-1})(A^\eps(x))-D(A^{-1})(A^\eps(x_n))\right)h}_H\,.
  \]
  Now, recalling {(H3)}, we have that $D(A^{\eps})(x)=\left(\eps I + D(A^{-1})(A^\eps(x))\right)^{-1}\in\cL(V,Y)$:
  hence, since $z\in V$, we infer that  
  \[
  h=D(A^\eps)(x)z \in Y \quad\text{and}\quad
  \norm{h}_Y\leq \norm{D(A^{\eps})(x)}_{\cL(V,Y)}\norm{z}_V\,.
  \]
  We deduce that 
  \[
  \norm{h_n-h}_H\leq \frac1\eps \norm{DA^{\eps}(x)}_{\cL(V,Y)}
  \norm{\left(D(A^{-1})(A^\eps(x))-D(A^{-1})(A^\eps(x_n))\right)}_{\cL(Y,H)}\norm{z}_V\,,
  \]
  hence also, from the arbitrariness of $z\in V$,
  \begin{align*}
  &\norm{D(A^\eps)(x_n)-D(A^\eps)(x)}_{\cL(V,H)}\\
  &\quad\leq \frac1\eps\norm{D(A^{\eps})(x)}_{\cL(V,Y)}
  \norm{\left(D(A^{-1})(A^\eps(x))-D(A^{-1})(A^\eps(x_n))\right)}_{\cL(Y,H)}\,.
  \end{align*}
  By the Lipschitz-continuity of $A^\eps$ we have that $A^\eps(x_n)\to A^\eps(x)$ in $H$ as $n\to\infty$, hence
  the right-hand side converges to $0$ as $n\to\infty$ again by {(H3)}.
\end{proof}

For every $\eps,\lambda>0$, 
we define the operator 
\[
  A_\lambda^\eps:= \lambda R + A^\eps_{|V}: V\to V^*\,.
\]
Since $A_\lambda^\eps$ is maximal monotone and coercive, its inverse
$(A_\lambda^\eps)^{-1}:V^*\to V$ is well-defined and Lipschitz-continuous.
We prove some properties of $A_\lambda^\eps$ in the following lemma.

\begin{lem}[Properties of $A_\lambda^\eps$]
  \label{lem:F_lam}
  Let $\eps,\lambda>0$ and define
  \[
  \varphi^\eps_\lambda:V\to[0,+\infty)\,, \qquad 
  \varphi^\eps_\lambda(x):=\frac\lambda2\norm{x}_V^2 + \varphi^\eps(x)\,, \quad x\in V\,.
  \]
  Then the convex conjugate $(\varphi_\lambda^\eps)^*$ of $\varphi^\eps_\lambda$, defined as
  \[
  (\varphi^\eps_\lambda)^*:V^*\to[0,+\infty)\,, \qquad 
  (\varphi^\eps_\lambda)^*(y):=\sup_{x\in V}\left\{\ip{y}{x}_V - \varphi^\eps_\lambda(x)\right\}\,, \quad y\in V^*\,,
  \]
  satisfies $(\varphi^\eps_\lambda)^* \in C^2(V^*)$, $D(\varphi_\lambda^\eps)^*=(A_\lambda^\eps)^{-1}$,
  and $D(\varphi^\eps_\lambda)^*$ and $D^2(\varphi^\eps_\lambda)^*$
  are locally bounded in $V^*$. Moreover, 
  $(A_\lambda^\eps)^{-1}:H\to V_0$ is G\^ateaux-differentiable and
  the following characterization holds:
  for every $y\in H$, setting $x_\lambda^\eps=(A_\lambda^\eps)^{-1}(y)$, we have
  \[
  D((A_\lambda^\eps)^{-1})(y)=
  \left[I + \lambda D((A^\eps)^{-1})(A^\eps(x_\lambda^\eps))\circ R\right]^{-1}\circ
  D((A^\eps)^{-1})(A^\eps(x_\lambda^\eps)) \in \cL(H,V_0)\,.
  \]
\end{lem}

\begin{proof}
  As $\lambda$ and $\eps$ are fixed throughout, we shall
  simplify notation and drop them in this proof.  
  By the classical results on the sum of subdifferentials
  \cite[Cor. 2.1, p. 41]{brezis}, we have that
  $A_\lambda^\eps=\lambda R + A^\eps_{|V} = \partial \varphi^\eps_\lambda$ so that   
  \[
  \partial (\varphi_\lambda^\eps)^*=(\lambda R + A^\eps_{|V})^{-1}:V^* \to V
  \]
  is Lipschitz-continuous.
  This implies that $(\varphi^\eps_\lambda)^*\in C^1(V^*)$ and 
  $D(\varphi^\eps_\lambda)^*=(\lambda R + A^\eps_{|V})^{-1}$
  is bounded on bounded subsets of $V^*$. Let us show now that 
  $(\lambda R + A^\eps_{|V})^{-1} \in C^1(V^*,V)$. 
  To this end, 
  we first note that the strong monotonicity of $A^\eps$ and the definition of 
  G\^ateaux derivative readily imply that 
  \[
  (DA^\eps(x)h,h)_H\geq \frac{c_A}2\norm{h}_H^2 \qquad\forall\,x,h\in H\,.
  \]
  Let $y\in V^*$ be arbitrary and set $x:=(\lambda R + A^\eps_{|V})^{-1}y\in V$.
  Since $R$ and $A^\eps_{|V}$ are Fr\'echet-differentiable in $x$ by {(P3)}, it is well-defined
  the operator
  \[
  D(\lambda R + A^\eps_{|V})(x)=\lambda R + DA^\eps_{|V}(x)\in \cL(V,V^*)\,:
  \] 
  let us show that it is an isomorphism from $V$ to $V^*$. It is clear that if $k\in V$ satisfies
  $(\lambda R + A^\eps_{|V}(x))k=0$, then
  \[
  \lambda\norm{k}^2_V + \frac{c_A}2\norm{k}_H^2\leq0\,,
  \]
  from which $k=0$, hence $\lambda R + DA^\eps_{|V}(x)$ is injective. Moreover, 
  since $k\mapsto \lambda Rk + DA^\eps_{|V}(x) k$ is linear,
  continuous, monotone, and coercive on $V$,
  it is also surjective, hence an isomorphism from $V$ to $V^*$.
  The theorem on differentiability of the inverse function yields then that 
  $(\lambda R + A^\eps_{|V})^{-1}$ is Fr\'echet-differentiable in $V^*$ and 
  \[
  D(\lambda R + A^\eps_{|V})^{-1}(y)=(\lambda R + DA^\eps_{|V}((\lambda R + A^\eps_{|V})^{-1}y))^{-1} 
  \qquad\forall\,y\in V^*\,.
  \]
  Let us show finally that $y\mapsto D(\lambda R + A^\eps_{|V})^{-1}(y)$ 
  is continuous from $V^*$ to $\cL(V^*, V)$.
  Let $y\in V^*$, $(y_n)_n\subset V^*$ with $y_n\to y$ in $V^*$, and 
  $h\in V^*$. Setting $k_n:=D(\lambda R + A^\eps_{|V})^{-1}(y_n)h$
  and $k:=D(\lambda R + A^\eps_{|V})^{-1}(y)h$, we have 
  \[
  \lambda R(k_n-k) + DA^\eps_{|V}((\lambda R + A^\eps_{|V})^{-1}y_n)k_n-
  DA^\eps_{|V}((\lambda R + A^\eps_{|V})^{-1}y)k=h-h=0\,,
  \]
  from which
  \[
  \lambda\norm{k_n-k}_V^2\leq
  \norm{DA^\eps_{|V}((\lambda R + A^\eps_{|V})^{-1}y_n)-DA^\eps_{|V}((\lambda R 
  + A^\eps_{|V})^{-1}y)}_{\cL(V,V^*)}
  \norm{k_n-k}_V\,.
  \]
  The Young inequality yields
  \[
  \lambda\norm{k_n-k}_V^2\leq
  \frac1\lambda\norm{DA^\eps_{|V}((\lambda R + A^\eps_{|V})^{-1}y_n)-
  DA^\eps_{|V}((\lambda R + A^\eps_{|V})^{-1}y)}_{\cL(V,V^*)}^2\,.
  \]
  Since $(\lambda R + A^\eps_{|V})^{-1}y_n\to (\lambda R + A^\eps_{|V})^{-1}y$ in $V$ and 
  $A^\eps_{|V}\in C^1(V,V^*)$ by (P3), we deduce that the right-hand side converges to $0$
  as $n\to\infty$, hence also
  \[
  \norm{D(\lambda R + A^\eps_{|V})^{-1}(y_n)-D(\lambda R + A^\eps_{|V})^{-1}(y)}_{\cL(V^*,V)}^2
  =\sup_{\norm{h}_{V^*}\leq 1}\norm{k_n-k}_{V}^2 \to 0
  \]
  as $n\to \infty$, from which $(\lambda R + A^\eps_{|V})^{-1} \in C^1(V^*,V)$.
  We deduce that $(\varphi^\eps_\lambda)^* \in C^2(V^*)$.
  Moreover, the fact that $(\lambda R + A^\eps_{|V})^{-1}$ is Lipschitz-continuous yields
  immediately that $D^2(\varphi^\eps_\lambda)^*$
  is bounded in $V^*$.

  Let us prove the last part of the lemma. Note that, for every $y\in H$, 
  setting $x_\lambda^\eps:=(A_\lambda^\eps)^{-1}(v)\in V_0$, we have
  $\lambda R x_\lambda^\eps + A^\eps(x_\lambda^\eps) = y$, from which
  $x_\lambda^\eps=\frac1\lambda R^{-1}(y-A^\eps(x_\lambda^\eps))$, so that 
  \[
  (A_\lambda^\eps)^{-1}(y) = \frac1\lambda R^{-1}(y - A^\eps((A_\lambda^\eps)^{-1}(y))) \qquad\forall\,y\in H\,.
  \]
  Since we already know that $y\mapsto (A_\lambda^\eps)^{-1}(y)\in C^1(H;V)$,
  recalling that $A^\eps_{|V}\in C^1(V;H)$ by {(P3)}
  and that $R^{-1}:H\to V_0$ is linear and continuous, 
  we infer that the operator
  \[
  y\mapsto (A_\lambda^\eps)^{-1}(y) = \frac1\lambda R^{-1}(y - A^\eps((A_\lambda^\eps)^{-1}(y)))
  \]
  is Fr\'echet-differentiable from $H$ to $V_0$. 
  Furthermore, since $x_\lambda^\eps = (A^\eps)^{-1}(y-\lambda R x_\lambda^\eps)$, we also have that 
  \[
  (A_\lambda^\eps)^{-1}(y) = (A^\eps)^{-1}(y-\lambda R (A_\lambda^\eps)^{-1}(y)) \qquad\forall\, y\in H\,.
  \]
  Since we have just proved that $y\mapsto (A_\lambda^\eps)^{-1}(y)$ is Fr\'echet-differentiable
  from $H$ to $V_0$, taking into account that $R:V_0\to H$ is linear continuous and that 
  $(A^{\eps})^{-1}$ is Lipschitz-continuous and G\^ateaux-differentiable from $H$ to $H$, we get
  \[
  D((A_\lambda^\eps)^{-1})(y) = D((A^\eps)^{-1})(y-\lambda R x_\lambda^\eps)\circ\left(I-
  \lambda R\circ D((A_\lambda^\eps)^{-1})(y)\right)\,,
  \]
  from which 
  \[
  \left[I + \lambda D((A^\eps)^{-1})(A^\eps(x_\lambda^\eps))\circ R\right]\circ D((A_\lambda^\eps)^{-1})(y)
  =D((A^\eps)^{-1})(A^\eps(x_\lambda^\eps))\,.
  \]
  Since $\left[I + \lambda D((A^\eps)^{-1})(A^\eps(x_\lambda^\eps))\circ R\right] \in \cL(V_0,H)$ is
  a linear isomorphism, we conclude.
\end{proof}

The next lemmata state some asymptotic properties 
of the operator $A_\lambda^\eps$ when $\lambda\searrow0$ and $\eps$ is fixed.
\begin{lem}
 \label{lem:as_lam}
 Let $y\in H$ and $\eps\in(0,c_A^{-1})$ be fixed. 
 For any $\lambda>0$ set $x_\lambda^\eps:=(A_\lambda^\eps)^{-1}(y)$:
 then, as $\lambda\searrow0$,
 \begin{align*}
 &x_\lambda^\eps \to (A^{\eps})^{-1}(y) \quad\text{in } H\,,\\
 &A^\eps(x_\lambda^\eps) \to y \quad\text{in } H\,,\\
 &D((A_\lambda^\eps)^{-1})(y) \wto D((A^\eps)^{-1})(y) \quad\text{in } \cL_w(H,H)\,.
 \end{align*}
\end{lem}
\begin{proof}
  Since $\lambda Rx_\lambda^\eps + A^\eps(x_\lambda^\eps)=y$, testing by 
  $x_\lambda^\eps$ and using {(P1)} we get
  \[
  \lambda\norm{x_\lambda^\eps}_V^2 + \frac{c_A}{2}\norm{x_\lambda}_H^2
  \leq\norm{y}_H\norm{x_\lambda^\eps}_H\leq 
  \frac{c_A}{4}\norm{x_\lambda^\eps}_H^2 + \frac1{c_A}\norm{y}_H^2\,,
  \]
  from which $\lambda x_\lambda^\eps\to 0$ in $V$. Moreover, testing by 
  $A^\eps(x_\lambda^\eps)$
  and using the monotonicity assumption in 
  {(H1)} we get $\norm{A^\eps(x_\lambda^\eps)}_H\leq\norm{y}_H$ for every $\lambda$, from which 
  the second convergence follows by the uniform convexity of $H$. The first convergence
  is then a consequence of the fact that $(A^\eps)^{-1}$ is Lipschitz-continuous. As for the third one, 
  let $h\in H$ be arbitrary, and set $k_\lambda^\eps:=D((A_\lambda^\eps)^{-1})(y)h$
  and $h^\eps_\lambda:=D((A^\eps)^{-1})(A^\eps(x_\lambda^\eps))h$, so that by Lemma~\ref{lem:F_lam}
  we have 
  \[
  k_\lambda^\eps + \lambda D((A^{\eps})^{-1})(A^\eps(x_\lambda^\eps)) Rk_\lambda^\eps = h_\lambda^\eps\,.
  \]
  Note that by definition of $A^\eps$ and {(H3)}, we have 
  \[
  h_\lambda^\eps=\eps h + D(A^{-1})(A^\eps(x_\lambda^\eps))h  
  \wto \eps h + D(A^{-1})(y)h=D((A^\eps)^{-1})(y)h \qquad\text{in } H\,,
  \]
  so that in particular $(k_\lambda^\eps)_\lambda$ is bounded in $H$.
  Hence, testing by $\lambda Rk_\lambda^\eps$ and employing the monotonicity 
  of $A^{-1}$ we get
  \[
  \lambda\norm{k_\lambda^\eps}_V^2 + \eps\lambda^2\norm{Rk_\lambda^\eps}_H^2
  \leq\lambda\norm{h_\lambda^\eps}_H\norm{Rk_\lambda^\eps}_H
  \leq \frac{\eps\lambda^2}{2}\norm{Rk_\lambda^\eps}_H^2 + \frac{1}{2\eps}\norm{h_\lambda^\eps}_H^2\,.
  \]
  Since $(h_\lambda^\eps)_\lambda$ is bounded in $H$, 
  we infer that $(\lambda k_\lambda^\eps)_\lambda$
  is bounded in $V_0$ and $(\lambda^{1/2}k_\lambda^\eps)_\lambda$
  is bounded in $V$: it follows that $\lambda k_\lambda^\eps
  \wto 0$ in $V_0$, hence in particular that $\lambda Rk^\eps_\lambda\wto 0$ in $H$.
  Since $D(A^{-1})\in C^0(H; \cL(V,H))$ and $D(A^{-1})(A^\eps(x_\lambda^\eps))$
  is symmetric, for every $z\in V$ we have
  \begin{align*}
  \left(\lambda D((A^{\eps})^{-1})(A^\eps(x_\lambda^\eps)) Rk_\lambda^\eps,z\right)_H=
  \eps\left(\lambda Rk_\lambda^\eps,z\right)_H + 
  \left(\lambda Rk_\lambda^\eps, D(A^{-1})(A^\eps(x_\lambda^\eps))z\right)_H\to 0\,,
  \end{align*}
  where we have used that $\lambda Rk_\lambda^\eps\wto0$
  and $D((A^{\eps})^{-1})(A^\eps(x_\lambda^\eps))z\to D((A^{\eps})^{-1})(y)z$ in $H$.
  Consequently, we have that 
  \[
  \lambda D(A^{-1})(A^\eps(x_\lambda^\eps))Rk_\lambda^\eps \wstarto 0 \qquad\text{in } V^*\,.
  \]
  Since we also have that $(k^\eps_\lambda)_\lambda$ is bounded in $H$,
  hence $k^\eps_\lambda\wto k^\eps$ in $H$ for a certain $k^\eps\in H$,
  letting $\lambda\searrow0$ we infer that $k^\eps=D((A^\eps)^{-1})(y)h$, and we conclude.
\end{proof}

\begin{lem}
  \label{lem:ident_gat}
  Let $\eps\in(0,c_A^{-1})$ and $y\in H$ be fixed. For any $\lambda>0$ let 
  $y_\lambda \in H$, and define $x^\eps:=(A^\eps)^{-1}(y)$
  and $x^\eps_\lambda:=(A_\lambda^\eps)^{-1}(y_\lambda)$.
  If $x^\eps_\lambda\to x^\eps$ in $H$ as $\lambda\searrow0$
  and $(y_\lambda)_\lambda$ is bounded in $H$,
  then it holds that, as $\lambda\searrow0$,
  \begin{align*}
  &A^\eps(x_\lambda^\eps)\to y \quad\text{in } H\,,\\
  &y_\lambda\wto y \quad\text{in } H\,,\\
  &D((A_\lambda^\eps)^{-1})(y_\lambda) \wto D((A^\eps)^{-1})(y) \quad
  \text{in } \cL_w(H,H)\,.
  \end{align*}
\end{lem}
\begin{proof}
  First of all, since $A^\eps$ is Lipschitz-continuous, we have $A^\eps(x_\lambda^\eps)\to A^\eps(x^\eps)=y$.
  Moreover, since $\lambda Rx_\lambda^\eps + A^\eps(x_\lambda^\eps)=y_\lambda$, testing by 
  $x_\lambda^\eps$ and using {(P1)} we get
  \[
  \lambda\norm{x_\lambda^\eps}_V^2 + \frac{c_A}{2}\norm{x_\lambda}_H^2
  \leq\norm{y_\lambda}_H\norm{x_\lambda^\eps}_H\leq 
  \frac{c_A}{4}\norm{x_\lambda^\eps}_H^2 + \frac1{c_A}\norm{y_\lambda}_H^2\,,
  \]
  from which $\lambda x_\lambda^\eps\to 0$ in $V$.
  Hence, by comparison in the equation and the boundedness of $(y_\lambda)_\lambda$
  in $H$ we infer that $y_\lambda\wto y$ in $H$, and the second convergence is proved.
  Let us show the last one.
  Let $h\in H$ be arbitrary, and set $k_\lambda^\eps:=D((A_\lambda^\eps)^{-1})(y_\lambda)h$
  and $h^\eps_\lambda:=D((A^\eps)^{-1})(A^\eps(x_\lambda^\eps))h$, so that by Lemma~\ref{lem:F_lam}
  we have 
  \[
  k_\lambda^\eps + \lambda D((A^{\eps})^{-1})(A^\eps(x_\lambda^\eps)) Rk_\lambda^\eps = h_\lambda^\eps\,,
  \]
  where by definition of $A^\eps$ and {(H3)}, we have 
  \[
  h_\lambda^\eps=\eps h + D(A^{-1})(A^\eps(x_\lambda^\eps))h  
  \wto \eps h + D(A^{-1})(y)h=D((A^\eps)^{-1})(y)h \qquad\text{in } H\,,
  \]
  so that in particular $(h_\lambda^\eps)_\lambda$ is bounded in $H$.
  Arguing as in the proof of Lemma~\ref{lem:as_lam} 
  we obtain that $k_\lambda^\eps\wto k^\eps$ and $\lambda Rk_\lambda^\eps\wto 0$
  in $H$ for a certain $k^\eps\in H$.
  Moreover, for every $z\in V$ we have 
  $D((A^{\eps})^{-1})(A^\eps(x_\lambda^\eps))z\to D((A^{\eps})^{-1})(y)z$ in $H$ thanks to {(H3)},
  so that by the symmetry of $D(A^{-1})(A^\eps(x_\lambda^\eps))$
  we get again
  \begin{align*}
  \left(\lambda D((A^{\eps})^{-1})(A^\eps(x_\lambda^\eps)) Rk_\lambda^\eps,z\right)_H=
  \eps\left(\lambda Rk_\lambda^\eps,z\right)_H + 
  \left(\lambda Rk_\lambda^\eps, D(A^{-1})(A^\eps(x_\lambda^\eps))z\right)_H\to 0\,,
  \end{align*}
  hence
  \[
  \lambda D(A^{-1})(A^\eps(x_\lambda^\eps))Rk_\lambda^\eps \wstarto 0 \qquad\text{in } V^*\,.
  \]
  By comparison in the equation we infer then that $k^\eps=D((A^\eps)^{-1})(y)h$, and we conclude.
\end{proof}

Finally, we prove a fundamental asymptotic property of $A_\lambda^\eps$
when $\eps=\lambda$ converge jointly to $0$.
To this end, we introduce for brevity of notation
the operator $\tilde A_\lambda:=A_\lambda^\lambda$ for any $\lambda>0$.
\begin{lem}
 \label{lem:as_lam'}
 \UUU
  Let $y\in H$, $x:=A^{-1}(y)\in H$, and 
 for any $\lambda>0$ set $x_\lambda:=\tilde A_\lambda^{-1}(y)\in V$,
 with $\tilde A_\lambda (w) := \lambda R w + A^\lambda(w)$ for any $w\in V$.
 Then, as $\lambda\searrow0$,
 it holds that $x_\lambda \wto x$ in $H$ and
 $A^\lambda(x_\lambda) \to y$ in $H$.
 Moreover, if $x\in V$ it also holds that $x_\lambda \to x$ in $V$ and
 \begin{align*}
   &D((\tilde A_\lambda)^{-1})(y) \wto D(A^{-1})(y) \quad\text{in } \mathscr L_w(H,H)\,.
 \end{align*}
\end{lem}
\begin{proof}
  Since $\lambda Rx_\lambda + A^\lambda(x_\lambda)=y$, testing by 
  $x_\lambda$ and using {(P1)} we get
  \[
  \lambda\norm{x_\lambda}_V^2 + \frac{c_A}{2}\norm{x_\lambda}_H^2
  \leq\norm{y}_H\norm{x_\lambda}_H\leq 
  \frac{c_A}{4}\norm{x_\lambda}_H^2 + \frac1{c_A}\norm{y}_H^2\,,
  \]
  from which $\lambda x_\lambda\to 0$ in $V$,
  so that $A^\lambda(x_\lambda)\to y$ in $V^*$. 
  Moreover, testing by 
  $A^\lambda(x_\lambda)$
  and using the monotonicity assumption in 
  {(H1)} we get $\norm{A^\lambda(x_\lambda)}_H\leq\norm{y}_H$ for every $\lambda$, from which 
  the second convergence follows by the uniform convexity of $H$. 
  The first convergence is then a consequence of the strong-weak closure of maximal monotone operators. 
  Let us now suppose that $x\in V$ and show that $x_\lambda\to x$ in $V$.
  From the relation $\lambda Rx_\lambda + A^\lambda(x_\lambda) - y = 0$,
  we test by $x_\lambda$ and rearrange the terms in the following way:
  \[
  \lambda\norm{x_\lambda}_V^2 + \left(A^\lambda(x_\lambda) - y, x_\lambda - (A^\lambda)^{-1}(y)\right)_H=
  \left(y-A^\lambda(x_\lambda), (A^\lambda)^{-1}(y)\right)_H\,.
  \]
  Noting that $y=A^\lambda((A^\lambda)^{-1}(y))$, the second term on the left-hand side is
  nonnegative; moreover, since $(A^\lambda)^{-1}(y)=\lambda y +A^{-1}(y)=\lambda y + x$, we get
  \[
  \lambda\norm{x_\lambda}_V^2\leq \left(y-A^\lambda(x_\lambda), \lambda y+x\right)_H
  \leq\lambda\norm{y}_H\norm{y-A^\lambda(x_\lambda)}_H + \left(y-A^\lambda(x_\lambda), x\right)_H\,.
  \]
  Recalling that $y-A^\lambda(x_\lambda)=\lambda Rx_\lambda$
  and using the Young inequality, we get
  \[
  \lambda\norm{x_\lambda}_V^2
  \leq\lambda\norm{y}_H\norm{y-A^\lambda(x_\lambda)}_H
  + \lambda\norm{x_\lambda}_V\norm{x}_V\leq
  \lambda\norm{y}_H\norm{y-A^\lambda(x_\lambda)}_H
  +\frac\lambda2\norm{x_\lambda}_V^2 + \frac\lambda2\norm{x}_V^2
  \]
  from which, dividing by $\lambda$ and rearranging the terms we get
  \[
  \norm{x_\lambda}_V^2\leq 2\norm{y}_H\norm{y-A^\lambda(x_\lambda)}_H + \norm{x}_V^2\,.
  \]
  Hence, recalling that $A^\lambda(x_\lambda)\to y$ in $H$, we infer that 
  \[
  \limsup_{\lambda\searrow0}
  \norm{x_\lambda}_V^2\leq 2\norm{y}_H\lim_{\lambda\searrow0}
  \norm{y-A^\lambda(x_\lambda)}_H + \norm{x}_V^2 = \norm{x}_V^2\,,
  \]
  from which we conclude by uniform convexity of $V$.

  \UUU
  Eventually, let us prove the last convergence \RRR of the
  statement. We \UUU 
  first note that for every $y_1,y_2\in H$,
  setting $x^i_\lambda:=\tilde A_\lambda^{-1}(y_i)$, for $i=1,2$, one has
  \[
  \lambda R(x_\lambda^1-x_\lambda^2) + A^\lambda(x_\lambda^1) - A^\lambda(x_\lambda^2) = y_1-y_2\,,
  \]
  so that testing by $x_\lambda^1-x_\lambda^2$ and
  exploiting the uniform strong monotonicity of
  $A^\lambda$ (see Lemma~\ref{lem:A_eps}), one deduces that 
  there exists $C>0$ independent of $\lambda$ such that 
  \[
  \|\tilde A_\lambda^{-1}(y_1)-\tilde A_\lambda^{-1}(y_2)\|_H
  \leq C\norm{y_1-y_2}_H \qquad\forall\,y_1,y_2\in H\,.
  \]
  It follows that $\|D((\tilde A_\lambda)^{-1})(y)\|_{\mathscr L(H,H)}\leq C$ 
  for every $\lambda>0$, hence that there 
  exists an  operator $L(y)\in\mathscr L(H,H)$ such that
  \[
    D((\tilde A_\lambda)^{-1})(y) \wto L(y) \quad\text{in } \mathscr L_w(H,H)\,.
  \]
 In order to complete the proof, we need to show that $L(y)h=D(A^{-1})(y)h$ for all $h\in H$:
  by density of $Z$ in $H$ it is enough to check this equality
  for $h\in Z$.
  We follow a similar argument as in Lemma~\ref{lem:as_lam} (where we take $\eps=\lambda$).
  For $h\in Z$ fixed, setting $k_\lambda:=D((\tilde A_\lambda)^{-1})(y)h$
  and $h_\lambda:=D((A^\lambda)^{-1})(A^\lambda(x_\lambda))h$, by Lemma~\ref{lem:F_lam}
  we have 
  \beq\label{aux_corr}
  k_\lambda + \lambda^2 Rk_\lambda + 
  \lambda D(A^{-1})(A^\lambda(x_\lambda)) Rk_\lambda = h_\lambda\,,
  \eeq
  where
  \[
  h_\lambda=\lambda h + D(A^{-1})(A^\lambda(x_\lambda))h \,.
  \]
  Since $(h_\lambda)_\lambda$ is clearly bounded in $H$, by testing \eqref{aux_corr}
  by $\lambda^2Rk_\lambda$ one readily gets 
  \beq
    \label{aux2_corr}
    \lambda\|R^{1/2}k_\lambda\|_H + \lambda^2\norm{Rk_\lambda}_H\leq C\,.
  \eeq
  Now, by interpolation we have
  \[
  \norm{R^{1-\eta}k_\lambda}_H\leq \|R^{1/2}k_\lambda\|_H^{2\eta}\norm{Rk_\lambda}_H^{1-2\eta}\,,
  \]
  so that \eqref{aux2_corr} yields also 
   \beq
    \label{aux3_corr}
    \lambda^{2(1-\eta)}\norm{R^{1-\eta}k_\lambda}_H\leq C\,.
  \eeq
  Moreover, by \RRR(H5) \UUU and the fact that $(x_\lambda)_\lambda$ is bounded in $V$,
  it follows that $(R^\eta h_\lambda)_\lambda$ is bounded in $H$: hence, 
  testing \eqref{aux_corr} by $\lambda^{2(1-\eta)}Rk_\lambda$ we get
  \[
  \lambda^{2(1-\eta)}\norm{k_\lambda}_V^2 + \lambda^{2(2-\eta)}\norm{Rk_\lambda}_H^2
  \leq (h_\lambda, \lambda^{2(1-\eta)}Rk_\lambda)_H = 
  \lambda^{2(1-\eta)}(R^\eta h_\lambda, R^{1-\eta}k_\lambda)_H\,.
  \]
  Taking \eqref{aux3_corr} into account we infer that 
  \beq
    \label{aux4_corr}
     \lambda^{1-\eta}\norm{k_\lambda}_V + \lambda^{2-\eta}\norm{Rk_\lambda}_H \leq C\,.
  \eeq
  Hence, again by interpolation we find that
  \beq
    \label{aux5_corr}
     \lambda^{2-3\eta}\norm{R^{1-\eta}k_\lambda}_H \leq C\,.
  \eeq
  Now, given $z\in Z$ arbitrary, one has that 
  $\|R^{\eta}D(A^{-1})(A^\lambda(x_\lambda))z\|_H\leq C\|z\|_Z$ by
  \RRR(H5) \UUU
  and the fact that $(x_\lambda)_\lambda$ is bounded in $V$, so that
  \eqref{aux5_corr} yields
  \begin{align*}
  \left(\lambda D(A^{-1})(A^\lambda(x_\lambda)) Rk_\lambda, z\right)_H
  =\left(\lambda R^{1-\eta}k_\lambda, R^{\eta}D(A^{-1})(A^\lambda(x_\lambda))z\right)_H
  \leq C\lambda^{3\eta-1}\norm{z}_Z\,.
  \end{align*}
  By recalling that $\eta\in(1/3,1/2)$ one finds 
  \[
  \norm{\lambda D(A^{-1})(A^\lambda(x_\lambda)) Rk_\lambda}_{Z^*}\leq C\lambda^{3\eta-1}\to 0\,,
  \]
  while \eqref{aux2_corr} yields directly 
  \[
  \norm{\lambda^2Rk_\lambda}_{V^*} \leq C\lambda\to 0\,.
  \]
  Passing to the weak limit in $Z^*$ as
  $\lambda\searrow0$ in equation \eqref{aux_corr}, 
  and noting that $h_\lambda\wto D(A^{-1})(y)h$ in $H$, 
  one obtains $L(y)h = D(A^{-1})(y)h$, and concludes.
  \EEE
\end{proof}


\section{A generalized It\^o's formula}
\label{sec:ito}

In this section we prove a generalized It\^o's formula that will be crucial 
in the proofs on the main results of the work. In particular, 
due to the weak assumptions on the derivatives of $\varphi^*$, we cannot 
rely directly on the classical frameworks by {\sc Da Prato \& Zabczyk}
\cite{dapratozab}
or {\sc Pardoux} \cite{Pard},
as the second derivative $D^2\varphi^*$ is assumed to exist in the sense of Fr\'echet
only in $V$ and the process $y$ below is not $V$-valued a priori.
This gives rise to several nontrivial difficulties.
Nevertheless, using the the continuity of the G\^ateaux derivative $D^2\varphi^*$ from $H$ to $\cL_w(H,H)$
and the fact that $y\in A(x)$ for a suitable $V$-valued process $x$,
we are able to show that the It\^o formula for $\varphi^*$ can still be written in an appropriate sense.

\begin{prop}[Generalized It\^o's formula]
  \label{prop:ito}
  Assume that $x$, $y$, $w$ are progressively measurable processes with values in $V$, $H$, and $V^*$,
  respectively, such that
  \begin{align*}
 & x \in L^0(\Omega; L^2(0,T; V))\,, \\
 & y\in L^0(\Omega; L^\infty(0,T; H)\cap C^0([0,T]; V^*))\,,
  \\
& w\in L^0(\Omega; L^2(0,T; V^*))\,,\\
 & y\in A(x) \quad\text{a.e.~in } \Omega\times(0,T)\,,\\
 & C\in L^0(\Omega; L^2(0,T; \cL^2(U,H)))\,,\\
 & y_0\in L^0(\Omega, \cF_0; H)\,,\quad
  x_0:=A^{-1}(y_0)\in L^0(\Omega,\cF_0; V)\,,\\
  &y(t)+\int_0^t w(s)\,ds = y_0  + 
  \int_0^tC(s)\,dW(s) \quad\text{in } V^* \quad\forall\,t\in[0,T]\,,\quad\P\text{-a.s.}
  \end{align*}
  Then, for every $t\in[0,T]$, $\P$-almost surely,
  \begin{align*}
  &\varphi^*(y(t)) + \int_0^t\ip{w(s)}{x(s)}_V\,ds \\
  &\quad= \varphi^*(y_0)
  +\int_0^t\left(x(s), C(s)\right)_H\,dW(s)
  +\frac12\int_0^t\operatorname{Tr}\left(C^*(s)D(A^{-1})(y(s))C(s)\right)\,ds\,.
  \end{align*}
\end{prop}

  Since $\partial\varphi^*=A^{-1}:H\to H$ is Lipschitz-continuous, we have in particular 
  that $H = D(\varphi^*)$: since $y\in C_w([0,T]; H)$ and $y_0$ is $H$-valued,
  the first terms on the left and right hand side in It\^o's formula are finite for every $t$.
  Moreover, 
  since $x=A^{-1}(y)$ and $A^{-1}$
  is Lipschitz-continuous, we also have 
  $x\in L^0(\Omega; L^\infty(0,T; H))$, hence
  $(x,C)_H\in L^2(0,T; \cL^2(U,\erre))$ and also the 
  stochastic integral on the right-hand side is well-defined.
  Finally, the trace term is also well-defined since $D(A^{-1})$ is
  bounded by assumption {(H3)}. 

\begin{proof}[Proof of Proposition \emph{\ref{prop:ito}}]
 For every $\lambda>0$
 by Lemma~\ref{lem:F_lam} we can apply 
 the classical It\^o's formula to 
 the function $\tilde\varphi_\lambda^*$,
 where $\tilde\varphi_\lambda:=\varphi_\lambda^\lambda$, getting
 \begin{align*}
    &\tilde \varphi_\lambda^*(y(t)) + \int_0^t\ip{w(s)}{\tilde A_\lambda^{-1}y(s)}_V\,ds\\
    &\quad=\tilde \varphi_\lambda^*(y_0) + \int_0^t\left(\tilde A_\lambda^{-1}y(s), C(s)\right)_H\,dW(s)
    +\frac12\int_0^t\operatorname{Tr}\left(C^*(s)D(\tilde A_\lambda^{-1})(y(s))
    C(s)\right)\,ds
  \end{align*}
  for every $t\in[0,T]$, $\P$-almost surely.
  By Lemma~\ref{lem:as_lam'} we have that 
  \begin{align*}
    &\tilde A_\lambda^{-1}y\to x \quad\text{in } L^2(0,T; V)\,,\\
    &D(\tilde A_\lambda^{-1})(y(s))\wto D(A^{-1})(y(s)) \quad\text{in } \cL_w(H,H) \quad\forall\,s\in[0,T]\,.
  \end{align*}
  Consequently, we have 
  \[
  \int_0^t\ip{w(s)}{\tilde A_\lambda^{-1}y(s)}_V\,ds \to 
  \int_0^t\ip{w(s)}{x(s)}_V\,ds
  \] 
  and, by the dominated convergence theorem, 
  \[
  \int_0^t\operatorname{Tr}\left(C^*(s)D(\tilde A_\lambda^{-1})(y(s))C(s)\right)\,ds\to
  \int_0^t\operatorname{Tr}\left(C^*(s)D(A^{-1})(y(s))C(s)\right)\,ds\,.
  \]
  Moreover, since $\partial\tilde \varphi_\lambda=\tilde A_\lambda$
  and $\partial\varphi^\lambda=A^\lambda$, for every $t\in[0,T]$
  we have
  \begin{align*}
  &\tilde \varphi_\lambda^*(y(t))=\left(v(t), \tilde A_\lambda^{-1}y(t)\right)_H 
  - \tilde \varphi_\lambda(\tilde A_\lambda^{-1}y(t))\\
  &\quad=\left(y(t), \tilde A_\lambda^{-1}y(t)\right)_H 
  - \frac\lambda2\norm{\tilde A_\lambda^{-1}y(t)}_V^2 - \varphi^\lambda(\tilde A_\lambda^{-1}y(t))\\
  &\quad=\left(y(t), \tilde A_\lambda^{-1}y(t)\right)_H 
  - \frac\lambda2\norm{\tilde A_\lambda^{-1}y(t)}_V^2\\
  &\qquad-\left(A^\lambda(\tilde A_\lambda^{-1}(y(t))), \tilde
    A_\lambda^{-1}(y(t))\right)_{H}
  + (\varphi^\lambda)^*\left(A^\lambda(\tilde A_\lambda^{-1}(y(t)))\right)\\
  &\quad=\left(y(t), \tilde A_\lambda^{-1}y(t)\right)_H 
  - \frac\lambda2\norm{\tilde A_\lambda^{-1}y(t)}_V^2
  -\left(A^\lambda(\tilde A_\lambda^{-1}(y(t))), \tilde
    A_\lambda^{-1}(y(t))\right)_{H} \\
  &\qquad+\frac\lambda2\norm{A^\lambda(\tilde A_\lambda^{-1}(y(t)))}_H^2
  +\varphi^*\left(A^\lambda(\tilde A_\lambda^{-1}(y(t)))\right)\,.
  \end{align*}
  Now, by Lemma~\ref{lem:as_lam'} we have that 
  \[
  \tilde A_\lambda^{-1}(y(t))\to x(t) \quad\text{in } H\,, \qquad
  A^\lambda(\tilde A_\lambda^{-1}(y(t))) \to y(t) \quad\text{in } H\,, \qquad
  \norm{\tilde A_\lambda^{-1}(y(t))}_V\leq M
  \]
  for a constant $M$ independent of $\lambda$. 
  Hence, as $\lambda\searrow0$ we get
  \[
  \tilde \varphi_\lambda^*(y(t)) \to
  \varphi^*(y(t))\,.
  \]
  Similarly, the same argument and the fact that $x_0=A^{-1}(y_0)\in V$ yields
  \[
  \tilde \varphi_\lambda^*(y_0)\to\varphi^*(y_0)\,.
  \]
  Finally, let us show the convergence of the stochastic integrals: note that
  \[
  \norm{((\tilde A^{-1}y-x),C)_H}^2_{L^2(0,T; \cL^2(U,\erre))}\leq
  \int_0^T\norm{\tilde A^{-1}y(s)-x(s)}_H^2\norm{C(s)}_{\cL^2(U,H)}^2\,ds
  \]
  where the integrand converges pointwise to $0$ in $[0,T]$.
  Moreover, from the proof of Lemma~\ref{lem:as_lam'} we infer that 
  \[
  \norm{\tilde A^{-1}y-x}_H^2\norm{C}_{\cL^2(U,H)}^2 \leq M
  \norm{C}^2_{\cL^2(U,H)}\left(\norm{y}_H^2+\norm{x}_H^2\right)
  \]
  for a positive constant $M$ only dependent on $c_A$.
  Since $x=A^{-1}(y)$ and $A^{-1}$ is Lipschitz-continuous,
  recalling that $y\in L^\infty(0,T; H)$ we deduce that also $x\in L^\infty(0,T; H)$,
  hence the dominated convergence theorem yields
  \[
  \norm{((\tilde A^{-1}y-x),C)_H}^2_{L^2(0,T; \cL^2(U,\erre))}\to 0
  \]
  $\P$-almost surely. We deduce that 
  \[
  \int_0^t\left(\tilde A_\lambda^{-1}y(s), C(s)\right)_H\,dW(s) \to 
  \int_0^t\left(x(s), C(s)\right)_H\,dW(s)
  \]
  in probability. Hence, the thesis follows letting $\lambda\searrow0$.
\end{proof}

The following result follows using exactly the same proof
with $A^\eps$ instead of $A$.
\begin{prop}[Generalized It\^o's formula for $A^\eps$]
  \label{prop:ito_eps}
  Let $\eps\in(0,c_A^{-1})$.
  Assume that $x^\eps$, $y^\eps$, $w^\eps$ 
  are progressively measurable processes with values in $V$, $H$, and $V^*$,
  respectively, such that
  \begin{align*}
  &x^\eps \in L^0(\Omega; L^2(0,T; V))\,, \\
  &y^\eps\in L^0(\Omega; L^\infty(0,T; H)\cap C^0([0,T]; V^*))\,,
  \\
  &w^\eps\in L^0(\Omega; L^2(0,T; V^*))\,,\\
  &y^\eps= A^\eps(x^\eps) \quad\text{a.e.~in } \Omega\times(0,T)\,,\\
  &C\in L^0(\Omega; L^2(0,T; \cL^2(U,H)))\,,\\
  &y_0^\eps\in L^0(\Omega, \cF_0; H)\,, \quad
  x_0^\eps:=(A^\eps)^{-1}(v_0^\eps)\in L^2(\Omega,\cF_0; V)\,,\\
  &y^\eps(t)+\int_0^tw^\eps(s)\,ds = y^\eps_0  + 
  \int_0^tC(s)\,dW(s) \quad\text{in } V^* \quad\forall\,t\in[0,T]\,,\quad\P\text{-a.s.}
  \end{align*}
  Then, for every $t\in[0,T]$, $\P$-almost surely,
  \begin{align*}
  &(\varphi^\eps)^*(y^\eps(t)) + \int_0^t\ip{w^\eps(s)}{x^\eps(s)}_V\,ds \\
  &\quad= (\varphi^\eps)^*(y_0^\eps)
  +\int_0^t\left(x^\eps(s), C(s)\right)_H\,dW(s)
  +\frac12\int_0^t\operatorname{Tr}\left(C^*(s)D((A^\eps)^{-1})(y^\eps(s))C(s)\right)\,ds\,.
  \end{align*}
\end{prop}


\section{Existence of martingale solutions: Proof of
  Theorem \ref{th:1}}
\label{sec:exist}
 
We introduce two separate approximations on the problem, 
depending on two different parameters $\lambda,\eps>0$,
prove uniform estimates on the regularized solutions
(Subsection \ref{s:1}), and 
pass to the limit as $\lambda\searrow0$ first
(Subsection \ref{s:2}) and then as $\eps\searrow0$
(Subsection \ref{s:3}).

\subsection{The approximated problem}\label{s:1}
Let $\eps\in(0,c_A^{-1})$ be fixed.
For every $\lambda>0$, we consider the approximated problem
\[
  d(A_\lambda^\eps u_\lambda^\eps) + B_\lambda u^\eps_\lambda \,dt =
  F(t,u_\lambda^\eps)\,dt + G(t,u_\lambda^\eps)\,dW\,, 
  \qquad A_\lambda^\eps u_\lambda^\eps(0)=v_0^\eps\,,
\]
where $B_\lambda:V\to V^*$ is the Yosida approximation of $B$
and $v_0^\eps:=(I+\eps R)^{-1}(v_0)$. 
We recall that since $V$ is a Hilbert space
then $R$ is linear and $B_\lambda$ is Lipschitz-continuous.
Setting $v_\lambda^\eps:=A_\lambda^\eps u_\lambda$, the approximated problem can be written in terms
of $v_\lambda^\eps$ as an evolution equation on $V^*$ of the form
\[
  d v_\lambda^\eps + B_\lambda (A_\lambda^\eps)^{-1}v^\eps_\lambda\,dt = 
  F(t,(A_\lambda^\eps)^{-1}v_\lambda^\eps)\,dt + 
  G(t,(A_\lambda^\eps)^{-1}v_\lambda^\eps)\,dW\,, \qquad v^\eps_\lambda(0)=v_0^\eps\,.
\]
Since $(A_\lambda^\eps)^{-1}:V^*\to V$
is Lipschitz-continuous, the operator 
$B_\lambda\circ(A_\lambda^\eps)^{-1}:V^*\to V^*$ 
is composition of 
Lipschitz-continuous operators, and there is a unique strong solution
\[
  v_\lambda^\eps \in L^q(\Omega; C^0([0,T]; V^*))
\]
such that, setting $u^\eps_\lambda:=(A_\lambda^\eps)^{-1}v^\eps_\lambda
 \in L^q(\Omega; C^0([0,T]; V))$, we have
\[
  v^\eps_\lambda(t) + \int_0^tB_\lambda u^\eps_\lambda(s)\,ds = v_0 +
  \int_0^tF(s,u_\lambda^\eps(s))\,ds + \int_0^t G(s, u_\lambda^\eps(s))\,dW(s) \quad\text{in } V^*
\]
for every $t\in[0,T]$, $\P$-almost surely. We now prove a priori
estimates, independent of $\lambda$ and $\eps$.

\begin{lem}[A priori estimates]\label{lem:est}
  Let $\eta\in(0,1/2)$. Then there exists a positive constant $M>0$ such that, 
  for any $\lambda>0$
  and $\eps\in(0,c_A^{-1})$,
  \begin{align*}
  &\lambda^{1/2}\norm{u_\lambda^\eps}_{L^q(\Omega; C^0([0,T]; V))} + 
  \eps^{1/2}\norm{A^\eps(u^\eps_\lambda)}_{L^q(\Omega; C^0([0,T]; H))}
    \leq M\,,\\
  &\norm{u_\lambda^\eps}_{L^q(\Omega; C^0([0,T]; H))} +
  \norm{A^\eps(u_\lambda^\eps)}_{L^q(\Omega; C^0([0,T]; H))} + 
  \norm{\varphi^*(A^\eps(u^\eps_\lambda))}_{L^{q/2}(\Omega;
    L^\infty(0,T))} \leq M\,,\\
  &\norm{J^B_\lambda u^\eps_\lambda}_{L^q(\Omega; L^2(0,T; V))} 
  + \norm{B_\lambda u^\eps_\lambda}_{L^q(\Omega; L^2(0,T; V^*))} \leq M\,,\\
  &\norm{G(\cdot, u_\lambda^\eps)}_{L^q(\Omega; C^0([0,T]; \cL^2(U,H)))}\leq M\,,\\
  &\norm{G(\cdot, u_\lambda^\eps)\cdot W}_{L^q(\Omega; W^{\eta,q}(0,T; H))}\leq M\,,\\
  &\norm{v^\eps_\lambda-G(\cdot, u_\lambda^\eps)\cdot W}_{L^q(\Omega; H^1(0,T; V^*))}
  \leq M\,,\\
  &\lambda\norm{Ru^\eps_\lambda}_{L^q(\Omega; W^{\eta,q}(0,T; V^*))} 
  + \norm{A^\eps(u^\eps_\lambda)}_{L^q(\Omega; W^{\eta,q}(0,T; V^*))} \leq M\,,
  \end{align*}
  where $J^B_\lambda:=(R+\lambda B)^{-1}\circ R:V\to V$ is the resolvent of $B$.
\end{lem}
\begin{proof}
  By Lemma~\ref{lem:F_lam}, we can apply
  It\^o's formula to $(\varphi^\eps_\lambda)^*(v^\eps_\lambda)$ in $V^*$, getting
  \begin{align*}
    &(\varphi^\eps_\lambda)^*(v^\eps_\lambda(t)) + 
    \int_0^t\ip{B_\lambda u^\eps_\lambda(s)}{u^\eps_\lambda(s)}_V\,ds
    =(\varphi^\eps_\lambda)^*(v_0) 
    + \int_0^t\left(F(s,u_\lambda^\eps(s)),u_\lambda^\eps(s)\right)_H\,ds \\
    &+ \int_0^t(u^\eps_\lambda(s),G(s,u_\lambda^\eps(s)))_H\,dW(s) + 
    \frac12\int_0^t\operatorname{Tr}\left(G^*(s,u_\lambda^\eps(s))
    (\lambda R + DA^\eps(u_\lambda^\eps(s)))^{-1}G(s,u_\lambda^\eps(s))\right)\,ds
  \end{align*}
  for every $t\in[0,T]$, $\P$-almost surely. Let us analyse the different terms separately.
  First of all, by definition of convex conjugate and the fact that 
  $\lambda Ru_\lambda^\eps + A^\eps(u^\eps_\lambda)=v^\eps_\lambda$ we have
  \begin{align*}
  &(\varphi^\eps_\lambda)^*(v^\eps_\lambda)=
  (\varphi^\eps_\lambda)^*(\lambda Ru^\eps_\lambda + A^\eps(u^\eps_\lambda))
  =\ip{\lambda Ru^\eps_\lambda + A^\eps(u^\eps_\lambda)}{u^\eps_\lambda}_V 
  - \varphi^\eps_\lambda(u^\eps_\lambda)\\
  &\quad=\lambda\norm{u^\eps_\lambda}_V^2 + (A^\eps(u^\eps_\lambda), u^\eps_\lambda)_H 
  - \frac\lambda2\norm{u^\eps_\lambda}_V^2 -\varphi^\eps(u^\eps_\lambda)\\
  &\quad=\frac\lambda2\norm{u^\eps_\lambda}_V^2 + (\varphi^\eps)^*(A^\eps(u^\eps_\lambda))
  =\frac\lambda2\norm{u^\eps_\lambda}_V^2 +
  \frac\eps2\norm{A^\eps(u^\eps_\lambda)}_H^2 + \varphi^*(A^\eps(u^\eps_\lambda))\,.
  \end{align*}
  Since $\varphi$ has at most quadratic growth in $H$, there are 
  constants $M', M''>0$, independent of $\lambda$ and $\eps$, such that
  \[
  \varphi^*(A^\eps(u^\eps_\lambda))\geq M'\norm{A^\eps(u_\lambda^\eps)}_H^2 - M''\,.
  \]
  Noting further that $u_\lambda^\eps=\eps A^\eps(u_\lambda^\eps) + A^{-1}(A^\eps(u_\lambda^\eps))$,
  we infer by the Lipschitz-continuity of $A^{-1}$ that 
  \[
  \varphi^*(A^\eps(u^\eps_\lambda))\geq 2M'c_A^{-1}\norm{u_\lambda^\eps}_H^2 - M''\,.
  \]
  Secondly, setting $u^\eps_{0\lambda}:=(\lambda R+A^\eps_{|V})^{-1}v_0^\eps$,
  a similar argument yields
  \[
  (\varphi^\eps_\lambda)^*(v_0)=\frac\lambda2\norm{u^\eps_{0\lambda}}^2_V 
  + (\varphi^\eps)^*(A^\eps(u^\eps_{0\lambda}))\,.
  \]
  Since $\lambda Ru^\eps_{0\lambda} + A^\eps(u^\eps_{0\lambda}) = v_0^\eps$, testing
  by $u^\eps_{0\lambda}$ and using the Young inequality we obtain
  \[
    \lambda\norm{u^\eps_{0\lambda}}_V^2 + \varphi^\eps(u^\eps_{0\lambda}) 
    + (\varphi^\eps)^*(A^\eps(u^\eps_{0\lambda}))
    =(v_0,u^\eps_{0\lambda})_H \leq \varphi^\eps(u^\eps_{0\lambda}) + (\varphi^\eps)^*(v_0^\eps)\,,
  \]
  where, by definition of $\varphi^\eps$,
  \[
  (\varphi^\eps)^*(v_0^\eps)=\frac\eps2\norm{v_0^\eps}_H^2+\varphi^*(v_0^\eps)
  \leq M(1+\norm{v_0^\eps}_H^2)\leq M(1+\norm{v_0}_H^2)\,.
  \]
  Hence,
  \[
  (\varphi^\eps_\lambda)^*(v_0^\eps)=
  \frac\lambda2\norm{u^\eps_{0\lambda}}_V^2 
  + (\varphi^\eps)^*(A^\eps(u^\eps_{0\lambda})) \leq (\varphi^\eps)^*(v_0)
  \leq M(1+\norm{v_0}_H^2) \in L^{q/2}(\Omega)\,.
  \]
  Moreover, hypothesis {\RRR(H7)\EEE} and the Young inequality immediately yield
  \[
  \int_0^t\left(F(s,u_\lambda^\eps(s)),u_\lambda^\eps(s)\right)_H\,ds\leq
  \frac14\norm{F(\cdot,0)}^2_{L^2(0,T; H)} + (L_F^2+1)\int_0^t\norm{u_\lambda^\eps(s)}_H^2\,ds\,,
  \]
  while
  by definition of $J^B_\lambda$ and coercivity of $B$ we have
  \begin{align*}
  &\ip{B_\lambda u^\eps_\lambda}{u^\eps_\lambda}_V=
  \ip{B_\lambda u^\eps_\lambda}{J^B_\lambda u^\eps_\lambda}_V
  +\ip{B_\lambda u^\eps_\lambda}{u^\eps_\lambda-J^B_\lambda u^\eps_\lambda}_V\\
  &\quad=\ip{B_\lambda u^\eps_\lambda}{J^B_\lambda u^\eps_\lambda}_V
  +\lambda\ip{B_\lambda u^\eps_\lambda}{R^{-1}B_\lambda u^\eps_\lambda}_V\geq
  c_B\norm{J^B_\lambda u^\eps_\lambda}_V^2 + \lambda\norm{B_\lambda u^\eps_\lambda}_{V^*}^2
  \end{align*}
  Furthermore, let $h\in H$ be arbitrary and set 
  $k^\eps_\lambda:=(\lambda R + DA^\eps(u^\eps_\lambda))^{-1} h\in V$,
  so that
  \[
  \lambda R k^\eps_\lambda + DA^\eps(u^\eps_\lambda)k^\eps_\lambda = h\,.
  \]
  Testing by $k^\eps_\lambda$ and using the strong
  monotonicity of $A^\eps$ we have
  \[
  \lambda\norm{k^\eps_\lambda}_V^2 + \frac{c_A}2\norm{k^\eps_\lambda}_H^2
   \leq (h,k^\eps_\lambda)_H\leq
  \frac{c_A}{4}\norm{k^\eps_\lambda}_H^2 + \frac1{c_A}\norm{h}_H^2\,.
  \]
  We infer that the following uniform estimate holds:
  \begin{equation}\label{est_der}
  \norm{(\lambda R + DA^\eps(u^\eps_\lambda))^{-1}}_{\cL(H,H)}^2\leq
  \frac4{c_A^2}\,.
  \end{equation}
  Hence, the trace term on the right-hand side of It\^o's formula can be estimated by
  \begin{align*}
  \operatorname{Tr}\left(G^*(\cdot, u_\lambda^\eps)
  (\lambda R + DA^\eps_{|V}(u^\eps_\lambda))^{-1}G(\cdot, u_\lambda^\eps)\right)
  &\leq\frac2{c_A}\norm{G(\cdot, u_\lambda^\eps)}^2_{\cL^2(U,H)}
  \leq\frac{4L_G^2}{c_A}\left(1 + \norm{u_\lambda^\eps}_H^2\right)\,.
  \end{align*}
  Finally, by the Burkholder-Davis-Gundy and Young inequalities we
  have, for every $\delta>0$,
  \begin{align*}
  &\E\sup_{t\in[0,T]}\left|\int_0^t(u^\eps_\lambda(s),G(s,u_\lambda^\eps(s)))_H\,dW(s)\right|^{q/2}
  \leq c
  \E\left(\int_0^T\norm{u^\eps_\lambda(s)}_H^2
  \norm{G(s,u_\lambda^\eps(s))}^2_{\cL^2(U,H)}\,ds\right)^{q/4}\\
  &\quad\leq\E\left(\norm{u^\eps_\lambda}_{L^\infty(0,T; H)}^{q/2}
  \norm{G(\cdot,u_\lambda^\eps)}_{L^2(0,T; \cL^2(U,H))}^{q/2}\right)\\
  &\quad\leq \delta \E\norm{u^\eps_\lambda}^q_{L^\infty(0,T; H)} 
  + \frac{1}{4\delta}\norm{G(\cdot,u_\lambda^\eps(s))}_{L^2(0,T; \cL^2(U,H))}^q\\
  &\quad\leq \delta \E\norm{u^\eps_\lambda}^q_{L^\infty(0,T; H)} 
  + \frac{L_G^2}{\delta}\left(1+\norm{u_\lambda^\eps}_{L^2(0,T;H)}^q\right)\,,
  \end{align*} 
  where $c$ is a positive constant, depending on data only. 
  Hence, recalling that $A$ and $B$ have linear growth, 
  taking supremum in time, power $q/2$ and expectations in It\^o's formula, 
  choosing $\delta$ small enough, rearranging the terms, and 
  employing the Gronwall lemma yield the first three desired
  estimates. Moreover, the fourth and fifth estimates easily
  follow from \RRR(H8) \EEE and the properties of the stochastic integral.

  Let us show now the last two estimates: since 
  $(B_\lambda u^\eps_\lambda)_{\lambda,\eps}$ 
  is uniformly bounded in the space $L^q(\Omega; L^2(0,T; V^*))$, we have
  \[
  \norm{\int_0^\cdot B_\lambda(u^\eps_\lambda(s))\,ds}_{L^q(\Omega; H^1(0,T; V^*))}\leq M
  \]
  for a positive constant $M$ independent of $\lambda$ and $\eps$. 
  Now, since $\eta\in(0,1/2)$ and $q>2$,
  we have that $1-1/2>\eta-1/q$, so
  the Sobolev 
  embeddings imply that 
  $H^1(0,T; V^*)\embed W^{\eta,q}(0,T; V^*)$ continuously, hence
  by comparison in the equation we have
  \[
    \norm{v^\eps_\lambda}_{L^q(\Omega; W^{\eta,q}(0,T; V^*))}\leq M\,.
  \]
  Now, recalling that 
  $v^\eps_\lambda=\lambda Ru^\eps_\lambda + A^\eps(u^\eps_\lambda)$, we have
  \[
  \E\int_0^T\!\int_0^T\frac{\norm{\lambda R(u^\eps_\lambda(s)-u^\eps_\lambda(r)) 
  + A^\eps(u^\eps_\lambda)(s)-A^\eps(u^\eps_\lambda)(r)}_{V^*}^q}
  {|s-r|^{1+\eta q}}\,ds\,dr\leq M\,.
  \]
  Since for almost every $s,r\in(0,T)$ we have, by monotonicity of $A^\eps$,
  \begin{align*}
  &\lambda^2\norm{R(u^\eps_\lambda(s)-u^\eps_\lambda(r))}_{V^*}^2
  =\ip{\lambda R(u^\eps_\lambda(s)-u^\eps_\lambda(r))}
  {\lambda(u^\eps_\lambda(s)-u^\eps_\lambda(r))}_V\\
  &\quad\leq
  \ip{\lambda R(u^\eps_\lambda(s)-u^\eps_\lambda(r)) 
  + A^\eps(u^\eps_\lambda)(s)-A^\eps(u^\eps_\lambda)(r)}{\lambda(u^\eps_\lambda(s)-u^\eps_\lambda(r))}_V\\
  &\quad\leq
  \frac12\norm{\lambda R(u^\eps_\lambda(s)-u^\eps_\lambda(r)) 
  + A^\eps(u^\eps_\lambda)(s)-A^\eps(u^\eps_\lambda)(r)}_{V^*}^2+
  \frac{\lambda^2}{2}\norm{R(u^\eps_\lambda(s)-u^\eps_\lambda(r))}_{V^*}^2\,,
  \end{align*}
  we deduce that $(\lambda Ru^\eps_\lambda)_\lambda$ is 
  uniformly bounded in the space $L^q(\Omega; W^{\eta,q}(0,T; V^*))$, 
  hence also $A^\eps(u^\eps_\lambda)
  =v^\eps_\lambda - \lambda Ru^\eps_\lambda$ by difference.
\end{proof}

\subsection{Passage to the limit as $\lambda\searrow0$}\label{s:2}
In this section we perform the passage to the limit as $\lambda\searrow0$,
while $\eps\in (0,c_A^{-1})$ is fixed.
We shall divide the passage to the limit is several steps.

\noindent{\bf Stochastic compactness.}
Fix $\eta\in(1/q, 1/2)$, which is possible since $q>2$. 
First of all, recalling Lemma~\ref{lem:est}, we have that the
families
$(A^\eps(u^\eps_\lambda))_\lambda$ and 
$(G(\cdot, u_\lambda^\eps)\cdot W)_\lambda$
are uniformly bounded in the space
\[
  L^q\left(\Omega;  W^{\eta,q}(0,T; V^*) \cap C^0([0,T]; H)\right)
\]
and the family $(W)_\lambda$ is constant in
$C^0([0,T];U)$. Since $H\cembed V^*$ compactly and $\eta q>1$, 
by the classical compactness results by Aubin-Lions and Simon
(see \cite[Cor.~5, p.~86]{simon}) we have the compact inclusion
\[
  W^{\eta,q}(0,T; H) \cap C^0([0,T]; H) \cembed C^0([0,T]; V^*)\,.
\]
This ensures by a standard argument based on the Markov inequality that 
the family of laws of $(A^\eps(u^\eps_\lambda))_\lambda$ and
$(G(\cdot, u_\lambda^\eps)\cdot W)_\lambda$ are tight
on the space $C^0([0,T]; V^*)$.
Secondly, it is clear that the laws of the constant sequences $(W)_\lambda$ and
$(v^\eps_0)_\lambda$
are tight on the spaces $C^0([0,T]; U)$ and $H$, respectively.

In particular, so is the family of laws of 
$(A^\eps(u^\eps_\lambda), G(\cdot, u_\lambda^\eps)\cdot W, W, v_0^\eps)_\lambda$
on the product space 
$C^0([0,T]; V^*)\times C^0([0,T]; V^*)\times C^0([0,T]; U)\times H$.
By Skorokhod's theorem (see \cite[Thm.~2.7]{ike-wata})
there exist a probability space $(\hat\Omega, \hat\cF, \hat\P)$, a family
$(\phi_\lambda)_\lambda$ of 
measurable mappings $\phi_\lambda:(\hat\Omega,\hat\cF)\to(\Omega,\cF)$ such that 
\[
  \P=\hat\P\circ\phi_\lambda^{-1} \qquad\forall\,\lambda>0\,,
\]
and measurable random variables 
$\hat v^\eps, \hat I^\eps:(\hat\Omega,\hat\cF)\to C^0([0,T]; V^*)$,
$\hat W^\eps:(\hat\Omega,\hat\cF)\to C^0([0,T]; U)$
and $\hat v_0^\eps:(\hat\Omega,\hat\cF)\to H$ such that, 
setting $\hat u_\lambda^\eps:=u^\eps_\lambda\circ\phi_\lambda$,
as $\lambda\searrow0$,
\begin{align*}
  A^\eps(\hat u^\eps_\lambda) \to \hat v^\eps 
  \qquad&\text{in } C^0([0,T]; V^*)\,, \quad\hat\P\text{-a.s.}\\
  \hat I_\lambda^\eps:=(G(\cdot, u_\lambda^\eps)\cdot W)\circ\phi_\lambda \to 
  \hat I^\eps 
  \qquad&\text{in } C^0([0,T]; V^*)\,, \quad\hat\P\text{-a.s.}\\
  \hat W_\lambda:=W\circ\phi_\lambda \to \hat W^\eps
  \qquad&\text{in } C^0([0,T]; U)\,, \quad\hat\P\text{-a.s.}\\
  \hat v_{0,\lambda}:=v_0^\eps\circ\phi_\lambda \to \hat v_0^\eps\qquad&\text{in } H\,, \quad\hat\P\text{-a.s.}
\end{align*}
Setting also $\hat v_\lambda^\eps :=v_\lambda^\eps\circ\phi_\lambda$, 
since $\P=\hat\P\circ\phi_\lambda^{-1}$ the uniform estimates given by Lemma~\ref{lem:est}
are preserved on the space $\hat\Omega$ for $(\hat u^\eps_\lambda)_\lambda$ and 
$(\hat v^\eps_\lambda)_\lambda$.
Consequently, there exist also two measurable random variables
$\hat u^\eps:(\hat\Omega,\hat\cF)\to L^2(0,T; V)$ and 
$\hat w^\eps:(\hat\Omega,\hat\cF)\to L^2(0,T; V^*)$ such that, as $\lambda\searrow0$,
\begin{align*}
  \lambda \hat u^\eps_\lambda \to 0 \qquad&\text{in } L^q(\hat\Omega; C^0([0,T]; V))\,,\\
  J^B_\lambda \hat u^\eps_\lambda\wto \hat u^\eps
  \qquad&\text{in } L^q(\hat\Omega; L^2(0,T; V))\,,\\
  B_\lambda \hat u^\eps_\lambda \wto \hat w^\eps
  \qquad&\text{in } L^q(\hat\Omega; L^2(0,T; V^*))\,,\\
  \hat v^\eps_\lambda, A^\eps(\hat u^\eps_\lambda)\to \hat v^\eps
  \qquad&\text{in } L^p(\hat\Omega; C^0([0,T]; V^*)) \quad\forall\,p\in[1,q)\,\\
  A^\eps(\hat u^\eps_\lambda)\wto \hat v^\eps \qquad&\text{in } L^q(\hat\Omega; L^2(0,T; H))\,,\\
  \hat I_\lambda^\eps\wstarto \hat I^\eps
  \qquad&\text{in } L^q(\hat\Omega; W^{1,\eta}(0,T; H))\,.
\end{align*}
Moreover, noting that $\hat u_\lambda^\eps-J_\lambda \hat u^\eps_\lambda
=\lambda R^{-1}B_\lambda \hat u^\eps_\lambda$,
it is immediate to deduce also that 
\[
  \hat u^\eps_\lambda\wto \hat u^\eps\qquad\text{in } L^{q}(\Omega; L^2(0,T; V))\,.
\]
Also, by lower semicontinuity of $\varphi^*$ and Lemma~\ref{lem:est} we deduce that
$\varphi^*(\hat v^\eps)\in L^{q/2}(\hat\Omega; L^\infty(0,T))$. Since {(H1)} implies that 
$\varphi^*$ is coercive on $H$ this immediately yields
\[
  \hat v^\eps\in L^q(\hat\Omega; L^\infty(0,T; H))\,.
\]
Moreover, since $A^\eps_{|V}:V\to2^{V^*}$ is maximal monotone by assumption {(H1)}, 
by strong-weak closure we immediately infer that 
\[
  \hat v^\eps = A^\eps(\hat u^\eps) \qquad\text{a.e.~in } \hat\Omega\times(0,T)\,.
\]
Now, noting that $\lambda R\hat u^\eps_\lambda + A^\eps(\hat u^\eps_\lambda)
= \hat v^\eps_\lambda$, it is clear that
\[
  \lambda R\hat u^\eps_\lambda + A^\eps(\hat u^\eps_\lambda) - \hat v^\eps = \hat v^\eps_\lambda -\hat v^\eps\,,
\]
hence testing by $\hat u^\eps_\lambda$ and using that $\hat v^\eps=A^\eps(\hat u^\eps)$ yields
\[
  \lambda\norm{\hat u^\eps_\lambda}_V^2 
  + (A^\eps(\hat u^\eps_\lambda) - \hat v^\eps , \hat u^\eps_\lambda)_H = 
  (\hat v^\eps_\lambda-\hat v^\eps, \hat u^\eps_\lambda)_H\,.
\]
Rearranging the terms and employing the strong monotonicity of $A^\eps$ yields
\begin{align*}
  &\lambda\norm{\hat u^\eps_\lambda}_V^2
  +\frac{c_A}2\norm{\hat u_\lambda^\eps-\hat u^\eps}_H^2\leq
  (\hat v_\lambda^\eps-\hat v^\eps, \hat u^\eps_\lambda)_H
  -(A^\eps(\hat u^\eps_\lambda)-\hat v^\eps,\hat u^\eps)_H\\
  &\quad\leq\norm{A^\eps(\hat u^\eps_\lambda)-\hat v^\eps}_{V^*}\norm{\hat u^\eps_\lambda}_V
  +\norm{A^\eps(\hat u^\eps_\lambda)-\hat v^\eps}_{V^*}\norm{\hat u^\eps}_V\,.
\end{align*}
Since $(\hat u^\eps_\lambda)_\lambda$ is bounded in $L^q(\hat\Omega; L^2(0,T; V))$
and $A^\eps(\hat u^\eps_\lambda) \to \hat v^\eps$,
$  \hat v^\eps_\lambda\to \hat v^\eps$
 in $L^p(\hat\Omega; L^2(0,T; V^*))$ for $p\in[1,q)$, we deduce also that
\[
  \hat u^\eps_\lambda \to \hat u^\eps \qquad\text{in } L^p(\hat\Omega; L^2(0,T; H)) \quad\forall\,p\in[1,q)\,,
\]
which implies by the Lipschitz-continuity of $F$ and $G$ that 
\begin{align*}
  F(\cdot, \hat u_\lambda^\eps) \to F(\cdot, \hat u^\eps)\qquad&\text{in } 
  L^p(\hat\Omega; L^2(0,T; H)) \quad\forall\,p\in[1,q)\,,\\
  G(\cdot, \hat u_\lambda^\eps) \to G(\cdot, \hat u^\eps)\qquad&\text{in } 
  L^p(\hat\Omega; L^2(0,T; \cL^2(U,H))) \quad\forall\,p\in[1,q)\,.
\end{align*}

\noindent{\bf Identification of the stochastic integral.}
By definition of $\hat I^\eps_\lambda:(\hat\Omega, \hat\cF)\to C^0([0,T]; V^*)$ we have that 
\[
  \hat I^\eps_\lambda= \hat v^\eps_\lambda + \int_0^\cdot B_\lambda \hat u^\eps_\lambda(s)\,ds 
  - \hat v_{0,\lambda} - \int_0^\cdot F(s,\hat u^\eps_\lambda(s))\,ds\,.
\]
By introducing the filtration
$(\hat\cF_{\lambda,t}^\eps)_{t\in[0,T]}$ as
\[
  \hat\cF_{\lambda,t}^\eps:=\sigma\{\hat u^\eps_\lambda(s),
  \hat I_\lambda^\eps(s), \hat W_\lambda(s): s\leq t\}\,, \qquad t\in[0,T]\,,
\]
one can show that 
$\hat I^\eps_\lambda$ is a square integrable 
$V^*$-valued martingale with respect to the filtration 
$(\hat\cF_{\lambda,t}^\eps)_{t\in[0,T]}$
with quadratic variation process given by
\[
  \qv{\hat I^\eps_\lambda}=\int_0^\cdot
  \norm{G(s, \hat u_\lambda^\eps(s))}_{\cL^2(U,V^*)}^2\,ds\,.
\]
Indeed, for any $s,t\in[0,T]$ with $s\leq t$ and for any real bounded continuous function
$g$ on $C^0([0,T]; V^*)$, recalling that $u^\eps_\lambda$ and $\hat u^\eps_\lambda$ have the same law, 
it is possible to see 
(for further details we refer here to \cite{fland-gat} and \cite[\S~8.4]{dapratozab}) that
\[
  \hat\E\left(\ip{\hat I^\eps_\lambda(t)-
  \hat I^\eps_\lambda(s)}{z}_Vg(\hat u^\eps_{\lambda|[0,s]})\right) = 0 \qquad\forall\,z\in V
\]
and
\begin{align*}
  &\hat\E\left[\left(\ip{\hat I^\eps_\lambda(t)}{z_1}_V\ip{\hat I^\eps_\lambda(t)}{z_2}_V-
  \ip{\hat I^\eps_\lambda(s)}{z_1}_V\ip{\hat I^\eps_\lambda(s)}{z_2}_V\right.\right.\\
  &\quad\left.\left.
  - \int_s^t\left(G(r, \hat u_\lambda^\eps(r))^*z_1, 
  G(r, \hat u_\lambda^\eps(r))^*z_2\right)_U\,dr\right)
  g(\hat u^\eps_{\lambda|[0,s]})\right] = 0 \qquad\forall\,z_1,z_2\in V\,.
\end{align*}
Hence, thanks to a classical representation theorem
for martingales (see \cite[Thm.~8.2]{dapratozab}), there exists a further
probability space that we can identify with no restrictions with 
$(\hat\Omega,\hat\cF,\hat\P)$,
such that (possibly enlarging the filtration $(\hat\cF_{\lambda,t}^\eps)_{t\in[0,T]}$)
\[
  \hat I^\eps_\lambda(t)=\int_0^tG(s, \hat u_\lambda^\eps(s))\,d\hat W_\lambda(s)\,, 
  \qquad t\in[0,T]\,.
\]
Now, passing to the limit as $\lambda\searrow0$ in the approximated equation we infer that
\[
  \hat v^\eps + \int_0^\cdot \hat w^\eps(s)\,ds =
  \hat v_{0}^\eps + \int_0^\cdot F(s,\hat u^\eps(s))\,ds
  + \hat I^\eps\,,
\]
while the boundedness and continuity of $g$ together with
the convergences obtained above imply that 
\[
  \hat\E\left(\ip{\hat I^\eps(t)-\hat I^\eps(s)}{z}_Vg(\hat u^\eps_{|[0,s]})\right) = 0 \qquad\forall\,z\in V
\]
and
\begin{align*}
  &\hat\E\left[\left(\ip{\hat I^\eps(t)}{z_1}_V\ip{\hat I^\eps(t)}{z_2}_V-
  \ip{\hat I^\eps(s)}{z_1}_V\ip{\hat I^\eps(s)}{z_2}_V\right.\right.\\
  &\quad\left.\left.
  - \int_s^t\left(G(r, \hat u^\eps(r))^*z_1, G(r,\hat u^\eps(r))^*z_2\right)_U\,dr\right)
  g(\hat u^\eps_{|[0,s]})\right] = 0 \qquad\forall\,z_1,z_2\in V\,.
\end{align*} 
Using now the strong convergences of $\hat W_\lambda$ to $\hat W^\eps$, 
$G(\cdot, \hat u_\lambda^\eps)$ to $G(\cdot, \hat u^\eps)$,
and $\hat I_\lambda^\eps$ to $\hat I^\eps$,
following the approach contained in \cite[\S~4.5]{SWZ18} (and the references therein)
one can show that, by possibly enlarging the probability space
$(\hat \Omega, \hat\cF, \hat\P)$,
there exists a complete right-continuous filtration
$(\hat\cF^\eps_t)_{t\in[0,T]}$ such that 
$\hat W^\eps$ is a 
cylindrical Wiener process
adapted to $(\hat\cF^\eps_t)_{t\in[0,T]}$ and
\[
  \hat I^\eps(t)=\int_0^tG(s,\hat u^\eps(s))\,d\hat W^\eps(s)\,, \qquad t\in[0,T]\,.
\]
Consequently, we deduce that 
\[
  \hat v^\eps(t) + \int_0^t\hat w^\eps(s)\,ds 
  = \hat v_{0}^\eps + \int_0^tF(s,\hat u^\eps(s))\,ds
  +\int_0^tG(s,\hat u^\eps(s))\,d\hat W^\eps(s) \quad\forall\,t\in[0,T]\,,\quad\hat\P\text{-a.s.}
\]

\noindent{\bf Identification of the nonlinearities.}
We have already proved that $\hat v^\eps=A^\eps(\hat u^\eps)$.
Let us show now that $\hat w^\eps\in B\hat u^\eps$ almost everywhere in $\hat\Omega\times(0,T)$.
To this end, recalling the proof of Lemma~\ref{lem:est} we have that 
\begin{align*}
    &\hat\E (\varphi^\eps_\lambda)^*(\hat v^\eps_\lambda(T)) + 
    \hat\E\int_0^T\ip{B_\lambda \hat u^\eps_\lambda(s)}{\hat u^\eps_\lambda(s)}_V\,ds\\
    &=\hat\E (\varphi^\eps_\lambda)^*(\hat v_0^\eps)
    +\hat\E\int_0^T\left(F(s,\hat u_\lambda^\eps(s)), \hat
      u_\lambda^\eps(s)\right)_H\,ds\\ 
      &+\frac12\hat\E\int_0^T\operatorname{Tr}
    \left(G^*(s,\hat u_\lambda^\eps(s))
    (\lambda R + DA^\eps(\hat u^\eps_\lambda(s)))^{-1}G(s,\hat u_\lambda^\eps(s))\right)\,ds\,,
  \end{align*}
with obvious meaning of the symbol $\hat\E$. Now, 
note that 
\[
  (\varphi^\eps_\lambda)^*(\hat v^\eps_\lambda(T)) =
  \frac\lambda2\norm{\hat u^\eps_{\lambda}(T)}_V^2 +
  (\varphi^\eps)^*(A^\eps(\hat u^\eps_\lambda(T))) \geq 
  (\varphi^\eps)^*(A^\eps(\hat u_\lambda^\eps(T)))\,,
\]
while from the proof of Lemma~\ref{lem:est} we also know that 
\[
  (\varphi^\eps_\lambda)^*(\hat v_0^\eps) 
  = \frac\lambda2\norm{\hat u^\eps_{0\lambda}}_V^2 + (\varphi^\eps)^*(A^\eps(\hat u_{0\lambda}^\eps))
  \leq (\varphi^\eps)^*(\hat v_0^\eps)\,.
\]
Moreover, since $A^\eps(\hat u^\eps_\lambda(T))\wto \hat v^\eps(T)$ in $L^2(\Omega; H)$, by weak
lower semicontinuity of $(\varphi^\eps)^*$ we have that
\begin{align*}
  &\limsup_{\lambda\searrow0}\hat\E\int_0^T\ip{B_\lambda \hat u^\eps_\lambda(s)}{\hat u^\eps_\lambda(s)}_V\,ds\\
  &\quad\leq\hat\E(\varphi^\eps)^*(\hat v_0^\eps) 
  -\liminf_{\lambda\searrow0}\hat\E(\varphi^\eps)^*(A^\eps(\hat u^\eps_\lambda(T)))
  +\lim_{\lambda\searrow0}\hat\E\int_0^T\left(F(s,\hat u_\lambda^\eps(s)), \hat
      u_\lambda^\eps(s)\right)_H\,ds\\
  &\qquad+\frac12\limsup_{\lambda\searrow0}\hat\E\int_0^T\operatorname{Tr}
    \left(G^*(s, \hat u^\eps_\lambda(s))
    (\lambda R + DA^\eps(\hat u^\eps_\lambda(s)))^{-1}
    G(s, \hat u^\eps_\lambda(s))\right)\,ds\\
 &\quad\leq\hat\E(\varphi^\eps)^*(\hat v_0^\eps) -\hat\E(\varphi^\eps)^*(\hat v^\eps(T))
 +\hat\E\int_0^T\left(F(s,\hat u^\eps(s)), \hat
      u^\eps(s)\right)_H\,ds\\
 &\qquad+\frac12\limsup_{\lambda\searrow0}\hat\E\int_0^T\operatorname{Tr}
    \left(G^*(s, \hat u^\eps_\lambda(s))
    (\lambda R + DA^\eps(\hat u^\eps_\lambda(s)))^{-1}
    G(s, \hat u^\eps_\lambda(s))\right)\,ds\,.
\end{align*}
Now, by Lemma~\ref{lem:ident_gat} and the strong-weak convergence, we deduce that 
\[
  \operatorname{Tr}
  \left(G^*(\cdot,\hat u^\eps_\lambda)
  (\lambda R + DA^\eps(\hat u^\eps_\lambda))^{-1}G(\cdot, \hat u^\eps_\lambda)\right) \to 
  \operatorname{Tr}
  \left(G^*(\cdot, \hat u^\eps)
  D((A^\eps)^{-1})(\hat v^\eps))G(\cdot, \hat u^\eps)\right)
\]
almost everywhere in $\hat\Omega\times(0,T)$.
Moreover, we have already proved that 
\[
  |\operatorname{Tr}\left(G^*(\cdot, \hat u^\eps)
  D((A^\eps)^{-1})(\hat v^\eps))G(\cdot, \hat u^\eps)\right)|
  \leq c \left(1+ \norm{\hat u_\lambda^\eps}^2\right)\,,
\]
where the right hand side is bounded in $L^{q/2}(\Omega\times(0,T))$
 and $c $ is a positive constant depending on data. Since $q/2>1$, 
 the right-hand side is uniformly integrable in $\Omega\times(0,T)$, 
hence so is the left-hand side, and Vitali's convergence theorem yields
\begin{align*}
  &\limsup_{\lambda\searrow0}\hat\E\int_0^T
  \ip{B_\lambda \hat u^\eps_\lambda(s)}{\hat u^\eps_\lambda(s)}_V\,ds
  \leq\hat\E(\varphi^\eps)^*(\hat v_0^\eps) -\hat\E(\varphi^\eps)^*(\hat v^\eps(T))\\
  &+\hat\E\int_0^T\left(F(s,\hat u^\eps(s)), \hat u^\eps(s)\right)_H\,ds
  +\frac12\hat\E\int_0^T\operatorname{Tr}
  \left(G^*(s,\hat u^\eps(s))D((A^\eps)^{-1})(\hat v^\eps(s))G(s,\hat u^\eps(s))\right)\,ds\,.
\end{align*}
Finally, by Proposition~\ref{prop:ito_eps} it is immediate to see that this implies
\[
  \limsup_{\lambda\searrow0}\hat\E\int_0^T\ip{B_\lambda \hat u^\eps_\lambda(s)}{\hat u^\eps_\lambda(s)}_V\,ds
  \leq \hat\E\int_0^T\ip{\hat w^\eps(s)}{\hat u^\eps(s)}_V\,ds\,.
\]
Hence, we infer that $\hat w^\eps\in B\hat u^\eps$ a.e.~in
$\hat\Omega\times(0,T)$ by \cite[Prop. 2.5, p. 27]{brezis}.

\subsection{The passage to the limit as $\eps\searrow0$}\label{s:3}
In this last section we perform the passage to the limit as $\eps\searrow0$.

First of all, by Lemma~\ref{lem:est}, the convergences proved in the previous section
and the weak lower semicontinuity of the norms, there exists a positive constant 
$M$, independent of $\eps$, such that 
\begin{align*}
  &\eps^{1/2}\norm{\hat v^\eps}_{L^q(\hat\Omega; L^\infty(0,T); H))} +
  \norm{\hat v^\eps}_{L^q(\Omega; W^{\eta,q}(0,T; V^*))} \leq M\,,\\
  &\norm{\hat u^\eps}_{L^q(\Omega; L^\infty(0,T; H))} + 
  \norm{\hat v^\eps}_{L^q(\Omega; L^\infty(0,T; H))} + 
  \norm{\varphi^*(\hat v^\eps)}_{L^{q/2}(\hat\Omega; L^\infty(0,T))} \leq M\,,\\
  &\norm{\hat u^\eps}_{L^q(\hat\Omega; L^2(0,T; V))} 
  + \norm{\hat w^\eps}_{L^q(\Omega; L^2(0,T; V^*))} \leq M\,,\\
  &\norm{G(\cdot, \hat u^\eps)}_{L^q(\Omega; C^0([0,T]; \cL^2(U,H)))}+
  \norm{G(\cdot, \hat u^\eps)\cdot \hat W^\eps}_{L^q(\Omega; W^{\eta,q}(0,T; H))}\leq M\,,\\
  &\norm{\hat v^\eps-G(\cdot, \hat u^\eps)\cdot \hat W^\eps}_{L^q(\Omega; H^1(0,T; V^*))}
  \leq M\,,
\end{align*}
Proceeding now as in the passage to the limit as $\lambda\searrow0$ in the previous 
section, using Skorokhod's theorem and the usual representation theorems
for martingales, we infer that there exist a further filtered probability space
and a cylindrical Wiener process on it, which we shall assume with no restriction
to coincide with $(\hat\Omega,\hat\cF,(\hat\cF)_{t\in[0,T]},\hat \P)$, and $\hat W$, 
respectively, such that 
\begin{align*}
  \hat v^\eps \to \hat v \qquad&\text{in } C^0([0,T]; V^*)\,, \quad\P\text{-a.s.}\,,\\
  \hat W^\eps \to \hat W
  \qquad&\text{in } C^0([0,T]; U)\,, \quad\P\text{-a.s.}\,,\\
  G(\cdot, \hat u^\eps)\cdot \hat W^\eps \to \hat I
  \qquad&\text{in } C^0([0,T]; V^*)\,,\quad\P\text{-a.s.}\,,\\
  \hat v_0^\eps \to \hat v_0 \qquad&\text{in } H\,, \quad\P\text{-a.s.}
\end{align*}
and
\begin{align*}
  \eps\hat u^\eps \to 0 \qquad&\text{in } L^q(\hat\Omega; L^\infty(0,T; H))\,,\\
  \hat u^\eps\wto \hat u\qquad&\text{in } L^q(\hat\Omega; L^2(0,T; V))\,,\\
  \hat w^\eps \wto \hat w\qquad&\text{in } L^q(\hat\Omega; L^2(0,T; V^*))\,,\\
  \hat v^\eps\to \hat v \qquad&\text{in } L^p(\Omega; L^2(0,T; V^*)) \quad\forall\,p\in[1,q)\,,\\
  \hat v^\eps\wto \hat v \qquad&\text{in } L^q(\hat\Omega; L^2(0,T; H))\,,\\
\end{align*}
Hence, the strong-weak closure of $A$ readily implies that $\hat v\in A(\hat u)$ almost 
everywhere in $\Omega\times(0,T)$.
Furthermore, since $\hat v^\eps = A^\eps(\hat u^\eps)$, we have that 
$\hat u^\eps=\eps \hat v^\eps + A^{-1}(\hat v^\eps)$, from which 
\[
  \hat v^\eps \in A(\hat u^\eps - \eps\hat v^\eps) 
  \qquad\text{and}\qquad\hat v \in A(\hat u)
\,.
\]
Consequently, the strong monotonicity of $A$ yields 
\begin{align*}
 & c_A\norm{\hat u^\eps-\hat u - \eps\hat v^\eps}_H^2 \leq
  \left(\hat v^\eps-\hat v, \hat u^\eps-\hat u - \eps\hat v^\eps\right)_H\\
  &\quad\leq
  \eps\norm{\hat v^\eps-\hat v}_H\norm{\hat v^\eps}_H 
  + \norm{\hat v^\eps-\hat v}_{V^*}\norm{\hat u^\eps - \hat u}_V\,,
\end{align*}
so that, integrating in time and recalling that $\eps\in(0,c_A^{-1})$ we have
\begin{align*}
  &\frac{c_A}{2}\norm{\hat u^\eps-\hat u}_{L^2(0,T; H)}^2\leq 
  c_A\norm{\hat u^\eps-\hat u-\eps\hat v^\eps}^2_{L^2(0,T; H)}
  +c_A\eps^2\norm{\hat v^\eps}_{L^2(0,T; H)}^2\\
  &\quad\leq \eps\norm{\hat v^\eps-\hat v}_{L^2(0,T; H)}\norm{\hat v^\eps}_{L^2(0,T; H)}
  +\norm{\hat v^\eps-\hat v}_{L^2(0,T; V^*)}\norm{\hat u^\eps-\hat u}_{L^2(0,T; V)}
  +\eps\norm{\hat v^\eps}_{L^2(0,T; H)}\,.
\end{align*}
For every $p\in[1,q)$, taking power $p/2$ at both sides and expectations, it follows
from the H\"older inequality that
\begin{align*}
  &\norm{\hat u^\eps-\hat u}_{L^p(\hat\Omega;L^2(0,T; H))}^p\leq 
  \left(\frac{2}{c_A}\right)^{p/2}\left(\eps^{p/2}\norm{\hat v^\eps-\hat v}_{L^p(\hat \Omega; L^2(0,T; H))}^{p/2}
  \norm{\hat v^\eps}_{L^p(\hat \Omega; L^2(0,T; H))}^{p/2}\right.\\
  &\quad\left.+\norm{\hat v^\eps-\hat v}_{L^p(\hat \Omega; L^2(0,T; V^*))}^{p/2}
  \norm{\hat u^\eps-\hat u}_{L^p(\hat \Omega; L^2(0,T; V))}^{p/2}
  +\eps^{p/2}\norm{\hat v^\eps}_{L^p(\hat \Omega; L^2(0,T; H))}^{p/2}\right)\,,
\end{align*}
from which, taking into account the already proved estimates, 
\[
  \norm{\hat u^\eps-\hat u}_{L^p(\hat\Omega;L^2(0,T; H))}^p\leq 
  M_p\left(\eps^{p/2} +
  \norm{\hat v^\eps-\hat v}_{L^p(\hat \Omega; L^2(0,T; V^*))}^{p/2} \right)\,,
\]
where $M_p$ is a positive constant independent of $\eps$.
Thanks to the strong convergence of $(\hat v^\eps)_\eps$, we get 
\[
  \hat u^\eps \to \hat u \qquad\text{in } L^p(\hat\Omega;L^2(0,T; H)) \quad\forall\,p\in[1,q)\,,
\]
from which 
\begin{align*}
  F(\cdot,\hat u^\eps) \to F(\cdot, \hat u) 
  \qquad&\text{in } L^p(\hat\Omega;L^2(0,T; H)) \quad\forall\,p\in[1,q)\,,\\
  G(\cdot,\hat u^\eps) \to G(\cdot, \hat u) 
  \qquad&\text{in } L^p(\hat\Omega;L^2(0,T; \cL^2(U,H))) \quad\forall\,p\in[1,q)\,,
\end{align*}
and arguing again as in the previous section, we have that 
\[
  \hat I= 
  G(\cdot, \hat u)\cdot \hat W \qquad\text{in } C^0([0,T]; V^*)\,, \quad\P\text{-a.s.}
\]

In order to conclude, we only need to prove that $\hat w\in B(\hat u)$ almost
everywhere in $\Omega\times(0,T)$. To this end, we recall the
It\^o formula for $A^\eps$ from Proposition~\ref{prop:ito_eps} and by
definition of $A^\eps$ and have
\begin{align*}
    &\frac\eps2\hat\E\norm{v^\eps(T)}_H^2 + \hat\E \varphi^*(\hat v^\eps(T)) 
    +\hat\E\int_0^T\ip{\hat w^\eps(s)}{\hat u^\eps(s)}_V\,ds\\
   & \quad=\frac\eps2\hat\E\norm{\hat v_0^\eps}_H^2 + \hat \E \varphi^*(\hat v_0^\eps)
    +\hat\E\int_0^T\left(F(s,\hat u^\eps(s)), \hat
      u^\eps(s)\right)_H\,ds\\
    &\quad+\frac\eps2\hat \E\int_0^T\norm{G (s,\hat u^\eps(s))}_{\cL^2(U,H)}^2\,ds+
    \frac12\hat\E\int_0^T\operatorname{Tr}
    \left(G^*(s,\hat u^\eps(s))(D(A^{-1})(\hat v^\eps(s)))G(s, \hat u^\eps(s))\right)\,ds\,,
\end{align*}
which by lower semicontinuity, assumption {(H4)}, and the
Dominated Convergence Theorem, implies
\begin{align*}
  &\limsup_{\eps\searrow0}\hat\E\int_0^T\ip{\hat w^\eps(s)}{\hat u^\eps(s)}_V\,ds
  \leq
  \hat \E \varphi^*(\hat v_0) -\hat\E \varphi^*(\hat v(T))\\
  &\quad+\E\int_0^T\left(F(s,\hat u(s)), \hat u(s)\right)_H\,ds
  +\frac12\hat\E\int_0^T\operatorname{Tr}
    \left(G^*(s, \hat u(s))(D(A^{-1})(\hat v(s)))G(s, \hat u(s))\right)\,ds\,. 
\end{align*}
We now use the It\^o formula for $A$ from
Proposition~\ref{prop:ito} in order to check that
\[
\limsup_{\eps\searrow0}\hat\E\int_0^T\ip{\hat w^\eps(s)}{\hat u^\eps(s)}_V\,ds\leq
\hat\E\int_0^T\ip{\hat w(s)}{\hat u(s)}_V\,ds\,.
\]
Owing to the latter, we conclude that  $\hat w\in B(\hat u)$ almost
everywhere in $\Omega\times(0,T)$ by \cite[Prop. 2.5, p. 27]{brezis}.


\section{Uniqueness and existence of strong solutions: Proof of
  Theorem \ref{th:2}}
\label{sec:uniq_strong}
We begin by showing uniqueness of martingale solutions on the same probability space.
Let $(u_1,v_1,w_1)$ and $(u_2,v_2, w_2)$ be two martingale solutions
to the problem \eqref{eq:0} on the same probability space. Then we have
\[
  d(v_1-v_2) + (w_1-w_2)\,dt = (F(\cdot,u_1)-F(\cdot, u_2))\,dt 
  +(G(\cdot,u_1)-G(\cdot, u_2))\,dW
  \,, \qquad (v_1-v_2)(0)=0\,.
\]

If $A$ is linear, continuous, and symmetric we have that 
\[
dA(u_1-u_2) + (w_1-w_2)\,dt = (F(\cdot,u_1)-F(\cdot, u_2))\,dt
+(G(\cdot,u_1)-G(\cdot, u_2))\,dW\,,
\]
so that It\^o's formula
and the Burkholder-Davis-Gundy inequality yield, for every $r\in[0,T]$,
\begin{align*}
  &\frac12\E\sup_{t\in[0,r]}\left(A(u_1-u_2))(t), (u_1-u_2)(t)\right)_H 
  + \int_0^r\ip{(w_1-w_2)(s)}{(u_1-u_2)(s)}\,ds \\
  &\leq  \E\int_0^r\left(F(s,u_1(s))-F(s,u_2(s)),(u_1-u_2)(s)\right)_H\,ds\\
  &\qquad+
  c  \E\left(\int_0^r\norm{G(s,(u_1(s))-G(s,u_2(s))}_{\cL^2(U,H)}^2
  \norm{(u_1-u_2)(s)}_H^2\,ds\right)^{1/2}\\
  &\qquad+
  \frac12\int_0^r\operatorname{Tr}\left(G^*(s,(u_1-u_2)(s))D(A^{-1})(A(u_1-u_2)(s))
  G(s,(u_1-u_2)(s))\right)\,ds
\end{align*}
where $c$ is a positive constant depending on data.
Using now the Lipschitz-continuity of $F$ and $G$, the boundedness of $D(A^{-1})$,
and the Young inequality, we deduce that for every $\delta>0$
\begin{align*}
  &\E\sup_{t\in[0,r]}\left(A(u_1-u_2))(t), (u_1-u_2)(t)\right)_H 
  + \int_0^r\ip{(w_1-w_2)(s)}{(u_1-u_2)(s)}\,ds \\
  &\leq  \delta \E\sup_{t\in[0,r]}\norm{(u_1-u_2)(t)}_H^2
  +C_\delta \E\int_0^r\norm{(u_1-u_2)(s)}_H^2\,ds\,.
\end{align*}
The monotonicity of $B$, the strong monotonicity and linearity of $A$
and the Gronwall Lemma imply that we can choose $\delta$
sufficiently small such that 
$\norm{(u_1-u_2)(t)}_H=0$ for every $t\in[0,T]$.
It follows that $u_1(t)=u_2(t)$ and $v_1(t)=v_2(t)$ for every $t\in[0,T]$,
hence also $w_1=w_2$ by comparison in the equation.

If $B$ is linear continuous and symmetric, we have 
\begin{align*}
(v_1-v_2)(t) + B\int_0^t(u_1-u_2)(s)\,ds 
&= \int_0^t(F(s,u_1(s))-F(s,u_2(s)))\,ds\\ 
&+\int_0^t\left(G(s,u_1(s))-G(s,u_2(s))\right)\,dW(s)
\end{align*}
for every $t\in[0,T]$.
Testing by $(u_1-u_2)(t)$, further integrating in time,
and using the 
Young inequality, yield, for all $r\in[0,T]$,
\begin{align*}
&\E\int_0^r\left((v_1-v_2)(t), (u_1-u_2)(t)\right)_H\,dt
+\frac12\E\ip{B\int_0^r(u_1-u_2)(s)\,ds}{\int_0^r(u_1-u_2)(s)\,ds}_V \\
&\leq\E \int_0^r\norm{\int_0^tF(s,u_1(s))-F(s,u_2(s))\,ds}_H\norm{(u_1-u_2)(t)}_H\,dt\\
&\qquad+\E\int_0^r\norm{\int_0^t\left(G(s,u_1(s))-G(s,u_2(s))\right)\,dW(s)}_H
\norm{(u_1-u_2)(t)}_H\,dt\\
&\quad\leq \frac{c_A}2\int_0^r \norm{(u_1-u_2)(t)}_H^2\,dt +
\frac{1}{c_A}\E\int_0^r\norm{\int_0^tF(s,u_1(s))-F(s,u_2(s))\,ds}_H^2\,dt\\
&\qquad+\frac{1}{c_A}\E\int_0^r\norm{\int_0^t\left(G(s,u_1(s))-G(s,u_2(s))\right)\,dW(s)}_H^2\,dt\,.
\end{align*}
Hence, the strong monotonicity of $A$, the monotonicity of $B$,
and the Lipschitz-continuity of $F$ and $G$ imply that
\[
  \frac{c_A}{2}\int_0^r\norm{(u_1-u_2)(t)}_H^2\,dt
  \leq
  \frac{L_F^2 + L_G^2}{c_A}\int_0^r\int_0^t\norm{(u_1-u_2)(s)}_H^2\,ds\,dt\,,
\]
so that $u_1=u_2$ on $[0,T]$ by the Gronwall Lemma, hence also $w_1=w_2$ by linearity. 
By comparison in the equation, also using the fact that $(v_1-v_2)(0)=0 $
it follows that $v_1=v_2$.

Finally, a classical argument shows that uniqueness of 
martingale solutions on the same probability space 
yields also existence (hence uniqueness) of a strong solution.
Indeed, this follows by a direct application of the following lemma,
due to {\sc Gy\"ongy \& Krylov} \cite[Lem.~1.1]{gyongy-krylov}.
\begin{lem}
  Let $\mathcal X$ be a Polish space and $(Z_n)_n$ be a sequence
  of $\mathcal X$-valued random variables. Then $(Z_n)_n$ converges
  in probability if and only if
  for any pair of subsequences $(Z_{n_k})_k$ and $(Z_{n_j})_j$, there exists 
  a joint sub-subsequence $(Z_{n_{k_\ell}}, Z_{n_{j_\ell}})_\ell$ converging 
  in law to a probability measure $\nu$ on $\mathcal X\times\mathcal X$ such
  that $\nu(\{(z_1,z_2)\in\mathcal X\times\mathcal X: z_1=z_2\})=1$.
\end{lem}
Going back to the proofs of Theorem~\ref{th:1}, it is not difficult to check 
that the Skorokhod theorem and the uniqueness of the limit problem
yield exactly the condition of the lemma above: see for example
\cite[\S~5]{vall-zimm}. 
Hence, one can recover
strong convergences of the approximating sequences
$(v^\eps_\lambda)_\lambda$ and $(v^\eps)_\eps$
in $C^0([0,T]; V^*)$ in probability also 
on the original probability space $(\Omega,\cF,\P)$.
The conclusion follows then by the same arguments 
on the space $(\Omega,\cF,\P)$, instead of $(\hat\Omega,\hat\cF,\hat\P)$.

\section{Applications}
\label{sec:appl}

We present now the application of the abstract
existence theory to nonlinear SPDE problems. In
particular, we will provide an existence result for martingale
solutions for relation \eqref{eq:00}, when complemented with
initial and boundary conditions.

\subsection{Doubly-nonlinear SPDEs with multivalued graphs}
In the following, let $\OO \subset \Rz^d$ $(d \in \Nz)$ be a nonempty,
open, bounded, and connected domain, with Lipschitz-continuous boundary
$\Gamma$.  For definiteness we shall let 
\[
V =H_0^1(\OO)\,, \qquad H=L^2(\OO)\,, \qquad V_0=H^2(\OO)\cap H^1_0(\OO)\,.
\]
which corresponds to consider homogeneous Dirichlet conditions. 
Note however that other classes of boundary conditions, 
including Neumann, Robin, and mixed, also of nonlinear type, may be considered as well, 
at the expense of minor notational modifications. 
Moreover, we let 
\[
Y=L^{r}(\OO)\,, \qquad\text{with }
\begin{cases}
  2< r \leq \displaystyle\frac{2d}{d-2} \quad&\text{if } d\geq3\,,\\
  2< r<+\infty \quad&\text{if } d=1,2\,.
\end{cases}
\]
This choice ensures that $V \subset Y \subset H$ densely and continuously.

The function $f: (0,T)\times \OO \times\erre\to \Rz$ is assumed to be of
Carath\'eodory type
  with $f(\cdot, \cdot,u)\in L^2((0,T)\times\OO)$
  for all $u\in \Rz$ and Lipschitz continuous in the variable $u$,
  uniformly in $(0,T)\times \OO$. In particular, by defining the Nemitzsky operator
$F(t,u)(x)=f(t,x,u(x))$, for a.e. $x \in \OO$, we have that \RRR(H7) \EEE
follows. Eventually, we assume to be given another separable Hilbert
space $U$ and require that the operator-valued function $G:(0,T)\times
L^2(\OO) \to
\cL^2(U,L^2(\OO))$ satisfies assumption \RRR(H8)\EEE. 

We now turn to the specification of classes of operators $A$ which
can be treated in our framework.  Let $\hat \alpha: \Rz \to \Rz$ be
convex, define $\alpha=\partial
  \hat \alpha : \Rz\to 2^\Rz$ and assume that $\alpha^{-1} \in
  C^1(\Rz)$, $0
  \in \alpha(0)$, and there exists $c_\alpha, \, C_\alpha>0$ such that 
  \begin{align*}
  &  (v_1- v_2)(u_1 - u_2) \geq c_\alpha |u_1 - u_2|^2 \quad \forall
    u_i \in \Rz, \ v_i \in \alpha(u_i), \nonumber\\
& |v| \leq C_\alpha (1+|u|) \quad \forall
    u \in \Rz, \ v \in \alpha(u)\,.\nonumber
  \end{align*}
Let us show that these positions entail the structural assumptions
(H1)-(H4). 

First of all, from the linear growth of $\alpha$ one has that
$\hat \alpha(u) \leq C (1+|u|^2)$ for some positive $C>0$. By defining $\varphi:H \to [0,\infty]$ as
$$\varphi(u)=\int_\OO \hat \alpha(u(x))\, dx$$
one has $\varphi$ is convex, everywhere defined (hence proper), 
lower semicontinuous, and $A=\partial \varphi : H
\to 2^H$  is maximal monotone with $0\in A(0)$. 
The strong monotonicity and the 
sublinearity $A$ follow from those of $\alpha$
with the choices $C_A = \sqrt{2}C_\alpha\, \min\{|\OO|^{1/2},1\}$ and
$c_A=c_\alpha$. In particular, we have that  
$D(A)=H\supset V$.  
Secondly, the $\eps$-Yosida regularization $A^\eps$ is given by $u \in H \mapsto
\alpha^\eps(u) = (u - ({\rm Id} + \eps\alpha)^{-1}(u))/\eps\in
H$. Note that $A^\eps(V)\subset V$ as $\alpha^\eps$ is
$(1/\eps)$-Lipschitz continuous. In particular, we have that 
$$(A^\eps(u),R(u))_H = \langle R(u), A^\eps(u) \rangle = \int_\OO
\nabla u \cdot \nabla (\alpha^\eps(u))\, dx = \int_\OO (\alpha^\eps)'(u)
|\nabla u|^2 \, dx \geq 0 \quad \forall u
\in V_0\,,$$
so that assumptions (H1)--(H2) are satisfied. 

Let us check (H3).
The operator $A^{-1}:H \to H$ is Lipschitz-continuous and $A^{-1}(v) = \alpha^{-1}(v)$, $v\in H$, a.e. in $\OO$. As
$\alpha^{-1}$ is $C^1$, $A^{-1}$ is differentiable and its
differential $D(A^{-1}): H \to \cL(H,H)$ is given by 
$$ D(A^{-1})(v)h = \gamma(v)h
\quad \text{a.e. in}  \ \OO, \quad \text{for} \ \
\gamma:=(\alpha^{-1})' \in C(\Rz)\cap L^\infty(\Rz).$$ 
Let now $v_n \to v$ in $H$. Then $\gamma(v_n)\to \gamma(v)$ in
$L^r(\OO)$ for all $r<\infty$ and $\gamma(v_n)\wstarto \gamma(v)$ in $L^\infty(\OO)$. We then have that
$$(D(A^{-1})(v_n)h_1,h_2)_H = \int_\OO \gamma(v_n)h_1h_2\, dx \to
\int_\OO \gamma(v)h_1h_2\, dx=  (D(A^{-1})(v)h_1,h_2)_H, $$
so that $D(A^{-1})\in C^0(H;\cL_w(H,H))$. Since $V \subset
Y=L^r(\OO)\subset H$, one can compute that 
\begin{align*}
&\|D(A^{-1})(v_n)  - D(A^{-1})(v)\|_{\cL(Y,H)}=\sup_{\|h\|_Y=1}
\|D(A^{-1})(v_n)h  - D(A^{-1})(v)h\|_H \\
&\quad \leq  \|\gamma(v_n) -
\gamma(v) \|_{L^{2r/(r-2)}(\OO)} \to 0.
\end{align*} 
We have hence checked that  $D(A^{-1})\in
C^0(H;\cL(Y,H))$ as well. 

Let now $x,\, v \in V$ be given and define $y =(I +
D(A^{-1})((I+A^{-1})^{-1}x))^{-1}v$, namely
$$ y + D(A^{-1})((I+A^{-1})^{-1}x)^{-1} y =v.$$
Test this equation on $|y|^{r-2}y$ and integrate on $\OO$ in order to
get
$$\| y \|_Y^{r} + \int_\OO \gamma((I+A^{-1})^{-1}x)|y|^{r}\, dx
\leq  \int_\OO |v|\, |y|^{r-1}\,  dx \leq \frac{1}{r}\| v
\|_{Y}^{r} + \frac{r}{r-1}\| y \|_{Y}^{r}$$
where we used the Young inequality. 
As $V \subset Y$ we conclude that
$\|y \|_Y \leq M\| v\|_V$ for $M>0$, as required.

We are hence left with checking (H4). Let then $u_\eps \wto u$ in
$V$ and $A^\eps (u_\eps) \wto v$ in $H$ with $v \in A(u)$:
we have to show that
\begin{equation}
\gamma'(\alpha^\eps(u_\eps))h\to \gamma'(v)h \quad\text{in } H \qquad\forall\,h\in H\,.
\label{inter}
\end{equation}
Assume at first that $\alpha$ is continuous, hence single-valued. Then 
$A^\eps(u_\eps) = A(J^A_\eps u_\eps) =
\alpha (J^A_\eps u_\eps) \to \alpha (u) = v$ a.e.~in $\OO$. In
particular, for all $h \in H$ we have that 
$$D(A^{-1})(A^\eps(u_\eps))h  = \gamma'(\alpha(J^A_\eps u_\eps))h \to
\gamma'(v)h = D(A^{-1})(v)h  \quad \text{in} \ H$$
by dominated convergence
and \eqref{inter} follows. Note that we need no Lipschitz-continuity
here. In fact, the Lipschitz-continuous case has been indeed
discussed in \cite{SWZ18}. 

The argument can however be extended to include the case of
non-single-valued graphs. For the sake of definiteness, let $\alpha = s
+ \tilde \alpha$, where $\tilde \alpha$ is monotone and continuous and $s$
is the {\it sign} graph $s(x) = x/|x|$ for $x \not =0$ and
$s(0)=[-1,1]$. Let us start by checking that 
$\gamma'(\alpha^\eps(u_\eps)) \to \gamma'(v)$ a.e.~in $\OO$. In order to prove
this, we use the notation $1_A$ for the indicator function of the
measurable set $A\subset\OO$, rewrite
$$\gamma'(\alpha^\eps(u_\eps)) = 
\gamma'(\alpha^\eps(u_\eps))  1_{\{u>0\}}+
\gamma'(\alpha^\eps(u_\eps))  1_{\{u<0\}}+\gamma'(\alpha^\eps(u_\eps)) 1_{\{u=0\}},$$
and discuss each term of this sum separately. As $J^\alpha_\eps u_\eps \to u$ a.e.~in $\OO$,
we have that 
\begin{align*}
&\lim_{\eps \to 0}\gamma'(\alpha^\eps(u_\eps))  1_{\{u>0\}} =
\lim_{\eps \to 0} \gamma' (\alpha(J^\alpha_\eps u_\eps))  1_{\{u>0\}}
= \lim_{\eps \to 0} \gamma' (1+\tilde\alpha(J^\alpha_\eps u_\eps))
1_{\{u>0\}} \\\
&\quad =\gamma' (1+\tilde\alpha(u))
1_{\{u>0\}}  = \gamma'(v) 1_{\{u>0\}} \quad \text{a.e.~in } \OO\,.
\end{align*}
Analogously, $\gamma'(\alpha^\eps(u_\eps))  1_{\{u<0\}} \to \gamma'(v)
1_{\{u<0\}} $ a.e.~in $\OO$. We now show that the remaining term on the set
$\{u=0\}$ is infinitesimal, namely $\gamma'(\alpha^\eps(u_\eps))
1_{\{u=0\}} \to 0$ a.e.~in $\OO$. Indeed, we can further decompose it as
$$\gamma'(\alpha^\eps(u_\eps))  1_{\{u=0\}} =
  \gamma'(\alpha^\eps(u_\eps)) 1_{\{u=0, \,|u_\eps|\leq \eps\}}+
    \gamma'(\alpha^\eps(u_\eps)) 1_{\{u=0, \,u_\eps >\eps\}} +
      \gamma'(\alpha^\eps(u_\eps)) 1_{\{u=0, \,u_\eps<-\eps\}}.$$
Since $\alpha^\eps(r) = r/\eps$ for $|r|\leq \eps$ and $\gamma'=0$ on
$[-1,1]$, the first term in the above right-hand side vanishes.
As for the remaining terms we argue as follows
\begin{align*}
 &\lim_{\eps \to 0} \big( \gamma'(\alpha^\eps(u_\eps)) 1_{\{u=0, \,u_\eps >\eps\}} +
      \gamma'(\alpha^\eps(u_\eps)) 1_{\{u=0, \,u_\eps<-\eps\}}\big)\\
&\quad =\lim_{\eps \to 0} \big(\gamma'(1+\tilde \alpha(J^\alpha_\eps u_\eps)) 1_{\{u=0, \,u_\eps >\eps\}} +
    \gamma'(-1+\tilde \alpha(J^\alpha_\eps u_\eps))  1_{\{u=0,
  \,u_\eps<-\eps\}}\big)\\
&\quad \leq \lim_{\eps \to 0} \big( \gamma'(1+\tilde
  \alpha(J^\alpha_\eps u_\eps))+ \gamma'(-1+\tilde
  \alpha(J^\alpha_\eps u_\eps)) \big)1_{\{u=0\}}\\
&\quad = \gamma'(1+\tilde
  \alpha(0)) +  \gamma'(-1+\tilde
  \alpha(0))  =\gamma'(1)+ \gamma'(-1)=0.
\end{align*}
We have hence proved that 
$$\gamma'(\alpha^\eps(u_\eps)) \to 
\left\{
  \begin{array}{ll}
\gamma'(v) &\text{on} \ \{ u \not = 0\}\\
0&\text{on} \ \{ u = 0\}
  \end{array}
\right\}\stackrel{v \in \alpha(u)}{\equiv} \gamma'(v)
\quad \text{a.e.~in } \OO\,.
$$
Since $\gamma'$ is bounded, dominated convergence entails \eqref{inter},
so that (H4) follows.

\UUU
As for \RRR(H5)\UUU, we note that $\alpha^{-1}$ is $0$ in $[-1,1]$, and has a 
$C^1$-extension in $\pm1$. If the extension 
in a right neighbourhood of $1$ and \RRR in \UUU a left-neighbourhood of $-1$ 
is of the form $|\cdot-1|^p$ and $|\cdot+1|^{p}$, respectively, for some $p>1$,
an easy computation shows that $(\alpha^{-1})'\circ \alpha$ is $(1-1/p)$-H\"older continuous.
Hence, $(\alpha^{-1})'\circ \alpha$ maps $H^1(\OO)$ into $H^s(\OO)$
for \RRR some \UUU $s\in(2/3, 1)$ for \RRR a \UUU suitable choice of
$p$. \RRR In particular, (H5) \UUU follows \RRR by choosing $Z=H^k(\OO)$ \UUU with $k\in\enne$ large enough.
\EEE

As for initial conditions, we require $v_0 \in
L^q(\Omega,\cF_0;L^2(\OO))$ with $\hat \alpha^*(v_0)
  \in L^{q/2}(\Omega,\cF_0;L^1(\OO))$, and $ u_0 := \alpha^{-1}(v_0)
  \in L^q(\Omega,\cF_0;H^1_0(\OO))$, which is noting but \RRR (H9)\EEE.

Let us now present a class of operators $B$ fitting our frame. 
Assume to be given $\beta_1: \Rz^d \to 2^{
\Rz^d}$ and $\beta_0:\Rz\to 2^{\Rz}$ maximal, monotone, and linearly
bounded. Note that $\beta_1$ is not required to be cyclic monotone. Moreover, we assume $\beta_1$ to be coercive, namely
$$\exists\, c_{\beta}>0 : \quad c_\beta|\xi|^2 \leq \eta\cdot \xi \quad
\forall \xi\in \Rz^d, \ \eta\in \beta_1(\xi).$$
We define $B:H^1_0(\OO) \to 2^{H^{-1}(\OO)}$ by letting $w \in B(u)$
iff there exist $\xi \in L^2(\OO,\Rz^d)$ and $b \in L^2(\OO)$ with
$\xi \in \beta_1(\nabla u)$ and $b \in \beta_0(u)$ a.e. such that 
$$\langle w, v \rangle = \int_\OO \xi \cdot \nabla v \, dx + \int_\OO
b\, v \, dx\quad \forall\,v \in H^1_0(\OO).$$
It is a standard matter to check that  $B$ is actually 
defined in all of $H^1_0(\OO)$, is maximal monotone, linearly bounded,
and coercive. In particular, \RRR(H6) \EEE holds.

Along with these positions, the abstract relation \eqref{eq:0}
corresponds to the variational formulation of \eqref{eq:00} under
homogeneous Dirichlet boundary conditions, namely
\begin{align}
 &d(\alpha(u)) - {\rm div} \, \beta_1(\nabla u)\, dt + \beta_0 (u)\,
  dt \ni f(u)\, dt + G(u)\, d W \quad \text{in}  \ \ H^{-1}(\OO), \
   \text{a.e. in $\Omega \times (0,T)$}\label{eq:a1}\\
& u = 0 \quad
  \text{on}  \ \ \partial \OO, \ \text{a.e. in $\Omega \times
  (0,T)$},\label{eq:a2}\\
& \alpha(u)(0) = v_0  \quad \text{in}  \ \ H^{-1}(\OO), \  \text{a.e. in $\Omega$}.\label{eq:a3}
\end{align}

Under the above assumptions, a direct application Theorem \ref{th:1}
entails the
existence of a martingale solution to \eqref{eq:a1}-\eqref{eq:a3}. 

\begin{thm}[Existence of a martingale solution to \eqref{eq:a1}-\eqref{eq:a3}]\label{th:3}
  There exists a quintuplet  
  $$((\hat\Omega, \hat\cF, (\hat\cF_t)_{t\in[0,T]}, \hat\P),\hat W, \hat u, \hat v, \hat w),$$
  where $(\hat\Omega, \hat\cF, \hat\P)$ is a probability space endowed with a filtration
  $(\hat\cF_t)_{t\in[0,T]}$ which is saturated and right-continuous, $\hat W$ is
  a cylindrical Wiener process, and $\hat u$, $\hat v$, and $\hat \xi$
  are progressively measurable processes with values in $H^1_0(\OO)$,
  $L^2(\OO)$, and  $L^2(\OO;\Rz^d)$, respectively, such that 
\begin{align*}
  &\hat u \in L^q(\hat\Omega; L^\infty(0,T; L^2(\OO))\cap L^2(0,T; H^1_0(\OO)))\,,\\
  &\hat v \in L^q(\hat\Omega; L^\infty(0,T; L^2(\OO))\cap C^0([0,T]; H^{-1}(\OO))\,,\\
  &\hat \xi\in L^q(\hat\Omega; L^2(0,T;  L^2(\OO;\Rz^d)))\,,\\
  &\hat v\in \alpha(\hat u)\,, \quad \hat \xi\in \beta (\nabla \hat u)
  \qquad\text{a.e.~in } \hat\Omega\times \OO \times (0,T)\,,\\
  &\hat v(t)-\int_0^t{\rm div}\,\hat \xi(s)\,ds = 
  \hat v(0) + 
  \int_0^t f(\hat u(s))\,ds+
  \int_0^t G(s,\hat u(s))\,d\hat W(s) \\
&\quad\text{in
  }H^{-1}(\OO)\ \ \forall\,t\in[0,T]\,,\ \hat\P\text{-a.s.}
  \end{align*}
  and $\hat v(0)$ has the same law of $v_0$ on $H^{-1}(\OO)$.
\end{thm}

\subsection{Stochastic equations with nonlocal terms}
Before closing this discussion let us mention that the abstract
existence result of Theorem \ref{th:1} applies to other classes of
SPDEs as well. As regards $A$, one could consider some linear
operators of positive type, even nonlocal in space. An  
example in this direction is $A = \partial \varphi$ for 
$$ \varphi(u) = \frac{c_A}{2}\int_\OO |u|^2\, dx + \frac12\int \!\! \int_{\OO\times\OO}
k(x,y)(u(x)-u(y))^2dx \, dy$$
where $c>0$ and $k\in L^2(\OO \times \OO)$ nonnegative and symmetric. In
particular, $\varphi$ is convex and lower semicontinuous and $Au(x)
= \partial \varphi(u)(x) = cu(x) + \int_\OO k(x,y)\, u (y) \ dy$ for
a.e. $x \in \OO$. Assumptions (H1)--(H3) are hence easy to check and
assumption (H4) would actually be not needed here, for $A$ is already
Lipschitz-continuous. 

Nonlocal operators $B$ could also be considered. A relevant example in
this direction is the fractional laplacian $(-\Delta)^r$ with $r\in
(0,1)$. When with homogeneous Dirichlet boundary conditions
in $\Rz^d\setminus \OO$, this can be variationally formulated
by letting $H=\{ v\in L^2(\Rz^d) \ | \ v =0 \ \text{a.e. in} \
\Rz^d\setminus \OO \}$ and
 defining the fractional Sobolev space \cite{hitchhiker}
\begin{align*}
  &V:=H^r_0(\OO)=\{ v\in H \ | \  [v]_{H^r(\OO)}<\infty\}
\end{align*}
where the {\it Gagliardo seminorm} $[v]_r$ reads
$$ [v]_{H^r(\OO)}:= \left(\int\!\!\int_{\OO \times \OO} \frac{|v(x)-v(y)|^2}{|x-y|^{d+2r}}\, dx \,
dy\right)^{1/2}.$$ 
When endowed with the norm $\| v \|_V^2:= \| v
\|_{L^2(\Rz^d)}^2+[v]^2_{H^r(\OO)}$, the Hilbert space $V \subset H$
densely and continuously. The subdifferential $B= \partial
[\cdot]^2_{H^r(\OO)}: V \to V^*$ is hence linear, positive, and
continuous and delivers a weak formulation of the
fractional laplacian $(-\Delta)^r$, up to a
multiplicative dimensional constant \cite{akagifrac}.

\subsection{Systems of doubly-nonlinear SPDEs}
Eventually, as $B$ is not required to be
cyclic monotone, one can tackle some classes of SPDE
systems as well. 
An example in this direction is
 \begin{align}
   &d( \vec \alpha( \vec u)) - {\rm Div} \, b_1(\nabla \vec u) \, dt   + \vec b_0
     (\vec u)\, dt\ni \vec F(\vec u)\,
     dt + \vec G (\vec u)
     \, d \vec W \quad \text{in} \  (H^{-1}(\OO))^n, \ \text{a.e. in} \
     \Omega \times (0,T)\label{eq:s}
 \end{align}
 where
$\vec u =(u_1,\dots,u_n): \Omega \times [0,T]\times \OO\to \Rz^n$, now with
$\vec u(t)\in V=(H^1_0(\OO))^n$ a.e.~in $\OO$. The maximal monotone graph 
$\vec \alpha: \Rz^n \to 2^{\Rz^n}$ is assumed to be diagonal, namely
$\vec\alpha(\vec u) = {\rm diag}(\alpha_1(u_1),\dots,\alpha_n(u_n))$, with all
$\alpha_i:\Rz \to 2^{\Rz}$ being of the type discussed above. The graphs
$b_1: \Rz^{n\times d} \to 2^{\Rz^{n\times d}}$ and $\vec b_0: \Rz^n
\to 2^{\Rz^n}$ are maximal monotone, possibly noncyclic, 
  linearly bounded and $b_1$
is coercive. Note that in the vectorial case ${\rm Div}$ is 
 the standard tensorial divergence, namely 
 $({\rm Div}\, b)_i= \sum_{j=1}^d\partial_j b_{ij}$. We define the operator $B:
V \to 2^{V^*}$ as $ w \in B(\vec u)$ iff $ \xi \in (L^2(\OO))^{n\times d}$
 and $\vec b \in (L^2(\OO))^{n}$ exist
 such that $ \xi \in b_1(\nabla  \vec u)$, $\vec b \in
 \vec b_0( \vec u)$ a.e.~in $\OO$, and 
$$\langle  w,  \vec z \rangle  =   \sum_{i=1}^n\int_\OO \vec\xi_i \cdot
\nabla z_i \,  d x  + \int_\OO 
\vec b \cdot  \vec z \, dx \quad \forall  z \in (H^1_0(\OO))^n$$
and remark that it
fulfills \RRR(H6) \EEE by not being cyclic monotone.
Asking $\vec u \mapsto \vec F(\vec u)$ to be Lipschitz
continuous, and $\vec G$ and $\vec W$ to be corresponding vectorial
versions of operator-valued coefficients and cylindrical Wiener
processes, the initial-value problem for the SPDE system \eqref{eq:s} can be variationally reformulated as   
 relation \eqref{eq:0} and the abstract Theorem \ref{th:1} provides the existence of
martingale solutions.


\section*{Acknowledgement}
This work has been partially supported by the Vienna Science and
Technology Fund (WWTF) through the project MA14-009 and by the
Austrian Science Fund (FWF) projects F\,65 and  I\,4354. \EEE

\bibliographystyle{abbrv}
\def\cprime{$'$}

\end{document}